\PassOptionsToPackage{margin=1in}{geometry}
\documentclass[11pt, a4paper]{article}
\usepackage[T1]{fontenc}
\usepackage[utf8]{inputenc}
\usepackage{microtype}
\usepackage{fullpage}
\usepackage[parfill]{parskip}
\usepackage{mathtools,amssymb,amsthm}
\usepackage[dvipsnames]{xcolor}
\usepackage{tikz}
\usepackage{pgfmath}

\newtheoremstyle{compact}
  {3pt}      
  {3pt}      
  {\itshape} 
  {}
  {\bfseries}
  {.}
  {.5em}
  {}

\theoremstyle{compact}

\newtheorem{theorem}{Theorem}[section]
\newtheorem{lemma}[theorem]{Lemma}
\newtheorem{definition}[theorem]{Definition}

\newtheorem{proposition}[theorem]{Proposition}

\newcounter{casenum}

\usepackage[a4paper]{geometry}
\usepackage{fancyhdr}
\usepackage{lastpage}
\usepackage{graphicx, wrapfig, subcaption, setspace, booktabs}
\usepackage[T1]{fontenc}
\usepackage[font=small, labelfont=bf]{caption}
\usepackage{fourier}
\usepackage[english]{babel}
\usepackage{sectsty}
\usepackage{url, lipsum}
\usepackage{hyperref}
\usepackage{xcolor}
\usepackage{listings}
\usepackage{color}
\usepackage[nottoc]{tocbibind} 

\definecolor{codegreen}{rgb}{0,0.6,0}
\definecolor{codegray}{rgb}{0.5,0.5,0.5}
\definecolor{codepurple}{rgb}{0.58,0,0.82}
\definecolor{backcolour}{rgb}{0.95,0.95,0.92}

\lstdefinestyle{mystyle}{
    backgroundcolor=\color{backcolour},   
    commentstyle=\color{codegreen},
    keywordstyle=\color{magenta},
    numberstyle=\tiny\color{codegray},
    stringstyle=\color{codepurple},
    basicstyle=\footnotesize,
    breakatwhitespace=false,         
    breaklines=true,                 
    captionpos=b,                    
    keepspaces=true,                 
    numbers=left,                    
    numbersep=5pt,                  
    showspaces=false,                
    showstringspaces=false,
    showtabs=false,                  
    tabsize=2
}
 
\usepackage{authblk}
\title{Caged subsequences in permutations}
\author[1]{Niranjan Balachandran\thanks{\tt{niranj@math.iitb.ac.in}}}
\author[2]{Omkar Ramdas\thanks{\tt{23d0787@iitb.ac.in, omkar.ramdas.02@gmail.com}}}
\author[3]{Umesh Shankar\thanks{\tt{204093001@iitb.ac.in, umeshshankar@outlook.com}}} 
\affil[1,2]{Department of Mathematics, Indian Institute of Technology Bombay, Powai, Mumbai 400076, India}
\affil[3]{Department of Computer Science and Automation, Indian Institute of Science Bengaluru, Bengaluru 560012, India} 

\date{August 10, 2026}
\begin{document}
\maketitle
\begin{abstract}
Given a sequence $\mathfrak{a}:=(a_1,\ldots,a_n)$ of reals, a subsequence $\mathfrak{b}=(a_{i_1},\ldots,a_{i_k})$ is said to be \emph{caged} if the largest and smallest among the members of $\mathfrak{b}$ are $a_{i_1}$ and $a_{i_k}$, though not necessarily in that order. In this paper, we consider the problem of maximal caged sequences in  permutations $\pi\in S_n$. We also consider the same problem for a random permutation, both when the permutation is chosen uniformly at random and also when it is picked uniformly at random from among the permutations of rectangular shape, via the RSK correspondence. We also give a $O(n\log n)$ time algorithm to obtain these maximal caged sequences. In addition, we propose a randomized linear time algorithm that gives the length of the maximal caged sequence up to the factor of $(1-o(1))$.
\end{abstract}

\section{Introduction}
Throughout the paper, $n$ will denote a positive integer and $[n]:=\{1,\ldots,n\}$ will denote the set of integers from $1$ to $n$. As usual, $S_n$ will denote the set of all permutations over $[n]$ and we will write a permutation $\pi$ as a sequence $(\pi(1), \pi(2), \dots, \pi(n))$, or simply as $(\pi_1,\ldots,\pi_n)$.\\

The study of occurrences of distinctive patterns in permutations is a rich area of Extremal Combinatorics, and also possibly counts among the oldest of extremal problems in Combinatorics. One such avenue of extremal problems concerns the occurrence of a \textit{pattern} in a permutation. Instead of trying to define this formally, we simply illustrate this notion through an example: a permutation $\pi=(\pi_1,\cdots,\pi_n)$ is said to have $123$ as a pattern if there exist $1\le i<j<k\le n$ such that $\pi_i<\pi_j<\pi_k$. Another way of stating the same would be: $\pi$ admits a monotone increasing subsequence of length $3$.\\

One of the earliest results about the unavoidability of certain patterns in permutations is the well-known Erd\H{o}s-Szekeres theorem \cite{ES}: Any sequence of $n^2+1$ real numbers contains a monotone subsequence of length $n+1$. Another way of stating the same as an extremal problem is the following: If $N(n)$ denotes the smallest integer such that any sequence of $N\ge N(n)$ reals contains a monotone subsequence of length $n$, then $N(n)=(n-1)^2+1$. The latter formulation is more precise since it also tells us that there are examples of sequences of length $(n-1)^2$ that do not admit a monotone subsequence of length $n$.\\

Another pattern whose occurrence in permutations is rather well-studied is an \textit{alternating subsequence}, viz., a subsequence $\pi_{i_1}\cdots \pi_{i_k}$ of $\pi$  satisfying $\pi_{i_1}>\pi_{i_2}<\pi_{i_3}>\pi_{i_4}<\cdots$. Again, one is interested in the length of a longest alternating sequence $as(\pi)$ in a given permutation $\pi$, and the number $as_k(\pi)$ of alternating sequences in $\pi$ of a fixed length of $k$. For more details, see \cite{Stan}). The study of patterns that necessarily arise in permutations can be further delved into with more restrictions and constraints on the patterns sought.\\

In this paper, we seek a different kind of pattern; in fact, we \textit{ask for less}.\\

Given a permutation $\pi\in S_n$, and a sequence $1\le i_1<i_2<\cdots<i_k\le n$, the subsequence $S = (\pi(i_1), \pi(i_2), \dots, \pi(i_k) )$ of $\, \pi\,$ is called \emph{caged} if $\, \{ \min(S), \max(S)\} = \{ \pi(i_1), \pi(i_k) \}$. For example, in the permutation
$$\pi=(8,9,5,6,14,13,2,16,15,1,4,3,11,12,7,10)$$ the subsequences $S_1=(8,5,6,2,1)$, $S_2=(8,9,14,13,16)$, $S_3=(1,4,3,11)$ are all caged subsequences. We shall call the caged sequence a \emph{Forward Caged Sequence (FCS)} if $\min(S)=\pi(i_1)$ and $\max(S)=\pi(i_k)$ and a \emph{Backward Caged Sequence (BCS)} if $\max(S)=\pi(i_1)$ and $\min(S)=\pi(i_k)$. In the example above, $S_1$ is backward caged, whereas $S_2$ and $S_3$ are forward caged. The length of any longest caged subsequence is called the caged length of $\pi$ and shall be denoted by $c(\pi)$.\\

The natural extremal question that arises is: Given $\pi\in S_n$, how large a caged sequence is it \textit{guaranteed} to contain? \\

Clearly, since any monotone subsequence of $\pi$ is also caged, the Erd\H{o}s-Szekeres theorem implies $c(\pi)$ is at least of the order $\sqrt{n}$. However, as can be seen in the example, none of $S_1,S_2,S_3$ are monotone sequences, so presumably, $c(\pi)$ could be much larger than the length of a largest monotone subsequence of $\pi$. In fact, it is not difficult to show that for any $\pi\in S_n$, $c(\pi)\ge\Omega(n)$. 

Before we outline that argument, here are a few simple observations. If $\pi^R$ denotes the \emph{reverse} permutation, i.e., $\pi^R(i)=\pi(n+1-i)$, then $c(\pi)=c(\pi^R)$.  Also, if $\overline{\pi}$ denotes the \emph{complement} permutation of $\pi$, i.e., $\overline{\pi}(i)=n+1-\pi(i)$, then we also have $c(\pi)=c(\overline{\pi})$. Furthermore, for $1\le i<j\le n$, $c_{\pi}[i,j]$ shall denote the length of a maximal caged sequence with the ends at positions $i,j$. We will often drop the subscript to keep the notation simpler.\\

Let $j_0:=\pi^{-1}(n)$, and without loss of generality, we may assume $j_0\ge n/2$, for otherwise we consider $\pi^R$. 
Let $BAD:=\{k\in(1,j_0): \pi(k)<\pi(1)\}$. Clearly, if $|BAD|\le j_0/2$ then $c[1,j_0]\ge j_0/2\ge n/4$, so without loss of generality assume that $n/4\le t=|BAD|>j_0/2$. Let the members of $BAD$ be $\ r_1,\ldots, r_t$ with $1<r_1<\cdots<r_t<j_0$. Let the minimum among $\{\pi(r_1),\ldots,\pi(r_t)\}$ occur at $r_h$ for some $h$. If $h<t/2$ then $c[r_h,j_0]\ge t/2$ and if $h\ge t/2$ then $c[1,r_h]\ge t/2$. In either case, we get an interval $[i,j]$ such that $c[i,j]\ge \frac{t}{2}\ge n/8$.\\

To find a permutation $\pi$ with $c(\pi)$ as small as possible, the natural heuristic is that if $i,j$ are far apart, then $\pi(i), \pi(j)$ must be close and vice-versa. This suggests the following `braiding' permutation as a good optimal candidate: 
Let $n = 4k + 1$. Set $\pi(2k+1)=1$. Next, place $4k + 1, 2$ to the right of $1$ and then $3, 4k$ to the left \emph{in that order}, and then again, switch to the right end with the next pair $4k-1, 4$ of elements. Repeat this braiding pattern till all the elements have been placed in some position, so that gives us the permutation:
 \[
 \begin{pmatrix}
    1 & 2 & 3 & 4 & \cdots & 2k -1 & 2k  & 2k + 1 & 2k + 2 & 2k + 3 & 2k + 4 & 2k + 5 & \cdots & 4k & 4k + 1 \\
   2k + 1 & 2k + 2 & 2k -1 & 2k + 4 & \cdots &  3 & 4k & 1 & 4k + 1 & 2 & 4k -1 & 4 & \cdots & 2k + 3 & 2k\\
  \end{pmatrix}
    \]

It is easy to see that for this $\pi$ we have $c(\pi)\sim n/4$, which makes the previous lower bound short by a factor of two.\\

Our first result shows that the heuristic is not correct at either end:
\begin{theorem}\label{caged}
    Consider $n = 5k $ for some positive integer $k$. For any $\pi \in S_n$, $\displaystyle c(\pi) \geq k + 2$. In general, $$\displaystyle \min_{\pi \in S_n} c(\pi) = \left\lfloor\frac{n}{5}\right\rfloor + 2 . $$
\end{theorem}
 The Erd\H{o}s-Szekeres theorem also has an asymmetric version - any sequence of $k\ge N(m,n):=mn+1$ reals admits either a monotone \emph{increasing} subsequence of length $m+1$ or a monotone \emph{decreasing} subsequence of length $n+1$, and again, the smallest $k$ for which this conclusion holds is $N(m,n)$. Naturally, one might ask the same for caged sequences: Given positive integers, $k,\ell$, determine the smallest integer $\mathcal{C}(k,\ell)$ such that any permutation $\pi\in S_n$ with $n\ge\mathcal{C}(k,\ell)$ admits a \emph{forward} caged sequence of length $k$ or a \emph{backward} caged sequence of length $\ell$. Of course, it is clear that such a $\mathcal{C}(k,\ell)$ exists since, if a sequence contains a caged sequence of length $\max(k,\ell)$ then surely one of the two aforementioned conclusions must hold. Indeed, an analogous result holds:
\begin{theorem}\label{asymmetric}
    For any positive integers $k,\, l \geq 3$,
    \[
    \mathcal{C}(k,l) = \mathcal{C}(l,k)  = \min \{ 3k + 2l, 2k + 3l\} - 10.
    \]
\end{theorem} 
Another highlight in the study of monotone subsequences in permutations is the strong connection with the RSK correspondence (see \cite{Sag} for a clear exposition). Every permutation $\pi$ is in bijective correspondence with a pair $(P(\pi),Q(\pi))$ of Standard Young Tableau (SYT) of the same shape, and the length of a longest monotone increasing subsequence of $\pi$ is equal to the length of the first row in $P(\pi)$ while the length of a longest decreasing subsequence in $\pi$ is equal to the length of the first column of $P(\pi)$. This leads to a natural question of considering a uniformly random permutation $\pi\in S_n$ and determining the asymptotics of the length $\ell(\pi)$ of a longest monotone subsequence in $\pi$. The celebrated work of Logan–Shepp \cite{LogShep} and Vershik–Kerov \cite{VerKer} determines the expected value of a longest monotone sequence in a very precise sense (also see \cite{Sur_math} for a simpler exposition). The paper of Stanley \cite{Stan} similarly addresses the question of a longest alternating sequence in a random permutation.\\

This leads us to similarly consider the same question for the parameter $c(\pi)$ for a random $\pi$. In fact, two natural questions are in order: 
\begin{enumerate}
    \item What is the behavior of $c(\pi)$ for a uniformly random $\pi\in S_n$?
    \item What is the behavior of $c(\pi)$ for a uniformly random $\pi\in S_{n^2}$ of \emph{square shape}? 
\end{enumerate}
The following theorem answers the first question.  
\begin{theorem}\label{expectation}
For any $\omega(n) \rightarrow \infty$ as $n \rightarrow \infty$,
\[
  \lim_{n \rightarrow \infty} \mathbb{P}\left( c(\pi) \geq \left(1-\frac{\omega(n)}{\sqrt{n}}\right)n \right) = 1.
\]
Also, 
\[
  \lim_{n \rightarrow \infty} \mathbb{P}\left( c(\pi) \leq \left(1-\frac{1}{\omega(n)\sqrt{n}}\right)n \right) = 1.
\]
Consequently, 
\[
\mathbb{E}[c(\pi)] = (1 - o(1))n .
\]
\end{theorem}
Theorem \ref{expectation} implies that for most permutations $\pi$, $c(\pi)$ is very close to $n$, so extremal examples such as the ones witnessing equality in theorem \ref{caged} are extremely rare. \\

As alluded to earlier, the RSK correspondence gives a very precise answer for the length of longest monotone subsequences in a given permutation. In particular, if the shape of the SYT in the RSK correspondence is fixed, the length of monotone subsequences are fixed as well. If $N=mn$, permutations with no monotone  increasing subsequences of length $n+1$ or decreasing subsequences of length $m+1$ (also called the Erd\H{o}s-Szekeres permutations) are precisely the permutations whose RSK shape is an $m\times n$ rectangle.   It is now a natural question to investigate how large $c(\pi)$ becomes if $\pi$ is a uniformly random Erd\H{o}s-Szekeres permutation. The next theorem addresses this particular question; it turns out that $c(\pi)$ is significantly lower than in the random case. 

\begin{theorem}\label{sqaure} Suppose $0<\theta\le 1$ be a fixed rational, and let $m=\theta n$.
    Let $\mathrm{ES}_{m,n} \subset S_{N} $ be the set of all Erd\H{o}s-Szekeres permutations for $N = mn$. Let $\pi$ be chosen uniformly at random from $\mathrm{ES}_{m,n}$. Then, there exists $ \varepsilon_0=\varepsilon(\theta) > 0$, such that  
    \[
    \lim_{n \rightarrow \infty} \mathbb{P}_{\pi \sim \mathrm{ES}_{m,n}}\left( c(\pi) \geq (1-\varepsilon_0)N \right) = 0.
    \]
    In the special case that $m=n$, with high probability, $c(\pi) \sim \beta N$ for 
    \[
    \beta \approx 0.732.
    \]
\end{theorem}

The final part of the paper considers the same question from a computational point of view: Given a permutation $\pi\in S_n$ how efficiently can we determine $c(\pi)$? Clearly, it can be determined in run time $O(n^3)$, and as observed by Y. Kasugai \cite{Yuuki},  it can be done in run time $O(n^2)$ as follows. For each pair $i <j$, let $u = \min\{\pi(i), \pi(j)\}$ and $v = \max\{\pi(i), \pi(j)\}$. Let $P(a,b)$ denote the number of indices $k\le a$ such that $\pi(k)\le b$. It is not hard to see that the table $P$ can be precomputed in $O(n^2)$ time. Then, for each pair $i < j$, $c[i,j] = P(j,v)-P(i-1,v)-P(j,u-1)+P(i-1,u-1)$, so each value $c[i,j]$ can be obtained in constant time. Since there are $O(n^2)$ pairs $(i,j)$, the maximum over all pairs, which gives $c(\pi)$ can be computed in time $O(n^2)$.The question that lurks now is: Is it possible to determine $c(\pi)$ in time $o(n^2)$?
\begin{theorem}\label{nlogn_algo} $c(\pi)$ can be computed in $O(n\log n)$. Further, there is a randomized algorithm that can determine $c(\pi)$ within an additive error of $O(\sqrt{n}\log n)$ in linear time.\end{theorem}
An $O(n\log n)$ algorithm for determining $c(\pi)$ was also independently discovered by \cite{Yuuki}.

The rest of the paper is organized as follows. We start with some preliminary results, and then proceed to prove theorems \ref{caged} and \ref{asymmetric} in the next two sections. The proofs of theorems \ref{expectation} and \ref{sqaure} will follow in the subsequent section, followed by the proof of theorem \ref{nlogn_algo} in the final section. We end with some concluding remarks and open questions.\\

\textbf{AI Usage declaration}: The authors declare that AI tools were used \textit{only} to check for references and prior work, \textit{but not in the discovery or 
proof of any of the results that appear in this paper}.

\section{Preliminaries}
Clearly, the maximum value that can be achieved by $c(\pi)$ is $n$. What can we say about the minimum value that it takes? We will answer this question in the next theorem. But before that, we will see some properties of the ceiling and floor functions ($\lfloor x \rfloor$ and $\lceil x \rceil$) for the quotients. 
\begin{lemma}
    The following holds for any positive integer $x$.
    \begin{enumerate}
        \item $\displaystyle \left\lceil \frac{x}{2} \right\rceil + \left\lfloor \frac{x}{2} \right\rfloor = x $,
        \item $\displaystyle 0 \leq \left\lceil \frac{x}{2}\right\rceil -  \left\lfloor \frac{x}{2} \right\rfloor \leq 1$,
        \item $\displaystyle \left\lceil\frac{1}{m} \left\lceil \frac{x}{n}\right\rceil \right\rceil = \left\lceil \frac{x}{mn} \right\rceil$ for all positive integers $m, n$, \label{floor_floor 3}
        \item $\displaystyle \left\lceil\frac{x}{2}\right\rceil - \left\lceil\frac{x}{4}\right\rceil = \left\lfloor\frac{x+1}{4}\right\rfloor.$  \label{floor_ceil_5}
    \end{enumerate}
\end{lemma}
These properties are easy to verify, so we omit their proofs. 

 In order to prove theorem \ref{sqaure}, we need some further notation and results. Recall that we can talk unambiguously about the shape of any permutation - this is just the shape of the SYT corresponding to $\pi$ under the RSK correspondence. \\
 
 For a rectangular-shaped permutation, it turns out that there is a much simpler way to read the permutation from the SYT pair $(P,Q)$, courtesy of \cite{Sur_math} :
 \begin{theorem}[General tableau sandwich theorem \cite{Sur_math}]\label{Gen_tab}
    Let $m, n \geq 1$. Denote by $\mathcal{T}\left(\square_{m, n}\right)$, the set of Young tableaux of shape $m \times n$. Also, let $\mathrm{ES}_{m,n}$ be the set of all permutations of $[mn]$ with shape $\square_{m, n}$. There is a bijection from $\mathcal{T}\left(\square_{m, n}\right) \times \mathcal{T}\left(\square_{m, n}\right)$ to $\mathrm{ES}_{m, n}$, described as follows. Given the tableaux $P=\left(p_{i, j}\right)_{\substack{1 \leq i \leq m \\ 1 \leq j \leq n}}, Q=\left(q_{i, j}\right)_{\substack{1 \leq i \leq m \\ 1 \leq j \leq n}} \in \mathcal{T}\left(\square_{m, n}\right)$, the permutation $\sigma \in \mathrm{ES}_{m, n}$ corresponding to the pair ( $P, Q$ ) satisfies
\begin{equation*}
\sigma\left(q_{i, j}\right)=p_{m+1-i, j}, \quad(1 \leq i \leq m, 1 \leq j \leq n) . 
\end{equation*}
\end{theorem}

The special case when $m=n$ is produced below for the sake of convenience.
\begin{theorem}[The tableau sandwich theorem for square-shaped permutations]\label{sandwich}
    \hspace{5mm} Let $\mathcal{T}_n$ be the set of square $n \times n$ standard Young tableaux. There is a bijection from $\mathcal{T}_n \times \mathcal{T}_n$ to $ E S_n$, defined as follows: to each pair of tableaux $P=\left(p_{i, j}\right)_{i, j=1}^n, \quad Q=\left(q_{i, j}\right)_{i, j=1}^n$ corresponds the permutation $\pi \in  E S_n$ given by
\begin{equation*}
\pi\left(q_{i, j}\right)=p_{n+1-i, j}, \quad(1 \leq i, j \leq n) .
\end{equation*}
\end{theorem}

It turns out that the asymptotics of the shape of a random square Young tableaux is reasonably well-understood, courtesy of \cite{lim_shapes}:
\begin{theorem}[Limit shape theorem for square Young tableaux]\label{limit_shape}
     Let $\mathcal{T}_n$ be the set of $n \times n$ square Young tableaux and let $\mathbb{P}_n$ be the uniform probability measure on $\mathcal{T}_n$. Then for the function $L:[0,1] \times[0,1] \rightarrow[0,1]$ defined below, we have: \\
     (Uniform convergence to the limit shape) For all $\epsilon>0$,
$$
\mathbb{P}_n\left(T \in \mathcal{T}_n: \max _{1 \leqslant i, j \leqslant n}\left|\frac{1}{n^2} t_{i, j}-L\left(\frac{i}{n}, \frac{j}{n}\right)\right|>\epsilon\right) \underset{n \rightarrow \infty}{\longrightarrow} 0 .
$$
\end{theorem}

The limit surface $L$ in the aforementioned theorem is described in terms of its level curves $\{  L(x, y) = \alpha \}$. We borrow the notation and terminology from \cite{Sur_math}. Consider the rotated coordinate system 
\begin{align*}
    u = \frac{x-y}{\sqrt{2}}, \hspace{5mm} v = \frac{x + y}{\sqrt{2}}. 
\end{align*}
\begin{figure}
\centering
    \begin{subfigure}[b]{0.41\textwidth}            
            \includegraphics[width=\textwidth]{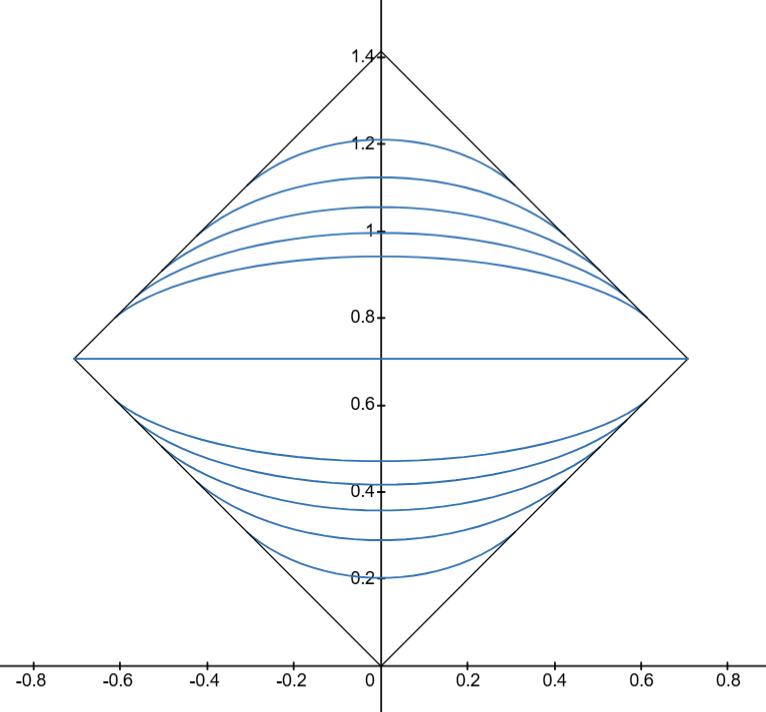}
            \caption{Rotated coordinate system}
            \label{fig:rot_lvl}
    \end{subfigure}%
    \hspace{7mm}
    \begin{subfigure}[b]{0.35\textwidth}
            \centering
            \includegraphics[width=\textwidth]{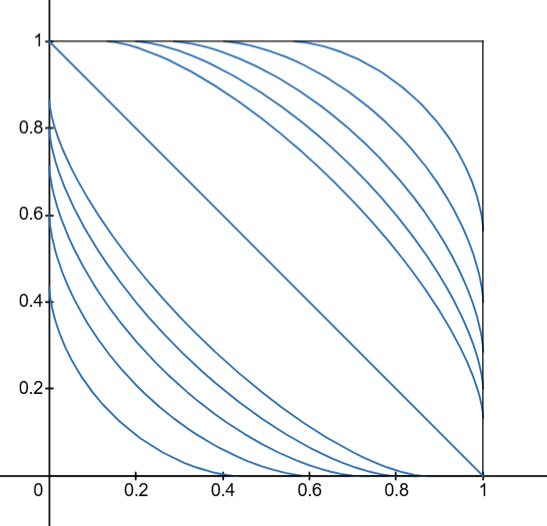}
            \caption{Usual coordinate system}
            \label{fig:usual}
    \end{subfigure}
    \caption{Level curves for limit surface}\label{fig:lvl}
\end{figure} 
Then in the $u-v$ plane, the square $[0,1] \times [0, 1]$ transforms into the rotated square:
\[
\diamond=\left\{(u, v) \in \mathbb{R}^2:|u| \leqslant \sqrt{2} / 2,|u| \leqslant v \leqslant \sqrt{2}-|u|\right\} .
\]
Now, define the one-parameter family of functions $\left(g_\alpha\right)_{0 \leqslant \alpha \leqslant 1}$ given by
\begin{align}
    g_\alpha: & [-\sqrt{2  \alpha(1-\alpha)},  \sqrt{2 \alpha(1-\alpha)}] \rightarrow \mathbb{R},   \label{eq: 6} \\
   & g_\alpha(u)= \begin{cases}\frac{2}{\pi} u \tan ^{-1}\left(\frac{(1-2 \alpha) u}{\sqrt{2 \alpha(1-\alpha)-u^2}}\right)+\frac{\sqrt{2}}{\pi} \tan ^{-1}\left(\frac{\sqrt{2\left(2 \alpha(1-\alpha)-u^2\right)}}{1-2 \alpha}\right) & 0 \leqslant \alpha<\frac{1}{2},  \\ -\frac{2}{\pi} u \tan ^{-1}\left(\frac{(2 \alpha-1) u}{\sqrt{2 \alpha(1-\alpha)-u^2}}\right)-\frac{\sqrt{2}}{\pi} \tan ^{-1}\left(\frac{\sqrt{2\left(2 \alpha(1-\alpha)-u^2\right)}}{2 \alpha-1}\right)+\sqrt{2} & \frac{1}{2}<\alpha \leqslant 1, \\ \frac{\sqrt{2}}{2} & \alpha=\frac{1}{2} .\end{cases} \notag
\end{align}
Then in the rotated coordinate system, the surface $\bar{L}(u, v)=L(x(u, v), y(u, v))$ can be described as the surface whose level curves $\{\bar{L}(u, v)=\alpha\}$ are exactly the curves $\left\{v=g_\alpha(u)\right\}$. That is,
$$
\{(u, v) \in \diamond: \bar{L}(u, v)=\alpha\}=\left\{(u, v) \in \diamond:|u| \leqslant \sqrt{2 \alpha(1-\alpha)}, v=g_\alpha(u)\right\} .
$$
See figure \ref{fig:lvl}. Note that we are following the French convention, so the level curves are increasing from bottom to top. \\

 Denote by $A_\pi$,  the graph of the permutation $\pi$, i.e., let $A_{\pi}:=\{(i,\pi(i)):i\in [n^2]\}$. Another important result from \cite{tab_san} that will be of importance for us is the following theorem about the shape of random square-shaped permutations.
\begin{theorem}\label{shape_square}
    Define the set
$$
\mathcal{Z}=\left\{(x, y) \in[-1,1] \times[-1,1]:\left(x^2-y^2\right)^2+2\left(x^2+y^2\right) \leq 3\right\}
$$

Then: (i) For any open set $U$ containing $\mathcal{Z}$,
$$
\mathbb{P}_n\left[\pi \in  \mathrm{ES}_n:\left(\frac{2}{n^2} A_\pi-(1,1)\right) \subset U\right] \xrightarrow[n \rightarrow \infty]{ } 1 .
$$
(ii) For any open set $U \subset \mathcal{Z}$,
$$
\mathbb{P}_n\left[\pi \in \mathrm{ES}_n:\left(\frac{2}{n^2} A_\pi-(1,1)\right) \cap U \neq \emptyset\right] \underset{n \rightarrow \infty}{\longrightarrow} 1 .
$$
\end{theorem}
 The first statement states that, as $n \rightarrow \infty$, a suitably scaled $A_\pi$ will be contained in $\mathcal{Z}$, while the second statement emphasizes that this containment will be dense in $\mathcal{Z}$. See figure \ref{fig:Z}. \\
\section{Proof of Theorem \ref{caged}}
\begin{proof}
    If $k=1, \, 2$ i.e., when $n = 5, \, 10$, the Erd\H{o}s-Szekeres theorem implies that any $\pi \in S_5$ (respectively, $S_{10}$) contains a monotone subsequence of length $3$ (respectively, $4$) and so the result follows immediately, so we shall assume that $k > 2$. \\
    
      Let $n = 5k$ and let $\pi \in S_n$. As we did in the introduction, 
     let $\pi(j_0) =n$ and $\displaystyle j_0 > \left\lfloor\frac{n}{2}\right\rfloor = \left\lfloor\frac{5k}{2}\right\rfloor $.
     For $i \in [1,j_0]$, define $rank(i)$ as 
    \[
    rank(i) := \#  \left\{ k \in [1,j_0]\, \big| \, \pi(k) \leq \pi(i)  \right\}. 
    \] 
    Let $\pi(i_0) = 1$. \\
    
    \textbf{Case 1: $i_0 > \displaystyle \left\lfloor\frac{n}{2}\right\rfloor$}. \\
     Since $j_0 \geq \displaystyle \left\lceil\frac{n}{2}\right\rceil$, there is an element with rank $\displaystyle \left\lceil \frac{1}{2} \left\lceil\frac{n}{2}\right\rceil \right\rceil = \left\lceil\frac{n}{4}\right\rceil$ (property \ref{floor_floor 3}). If $rank(1) > \displaystyle \left\lceil\frac{n}{4}\right\rceil$, then note that $c[1,i_0] \geq \displaystyle \left\lceil\frac{n}{4}\right\rceil + 1$, and since $ \displaystyle \left\lceil\frac{n}{4}\right\rceil + 1 = \left\lceil\frac{5k}{4}\right\rceil + 1 \geq k + 2$ for all positive integers $k$, we are done in this case. so suppose $rank(1) = r\leq  \displaystyle \left\lceil\frac{n}{4}\right\rceil$.\\
     \begin{figure}[h]
        \centering
        \includegraphics[width=1\textwidth]{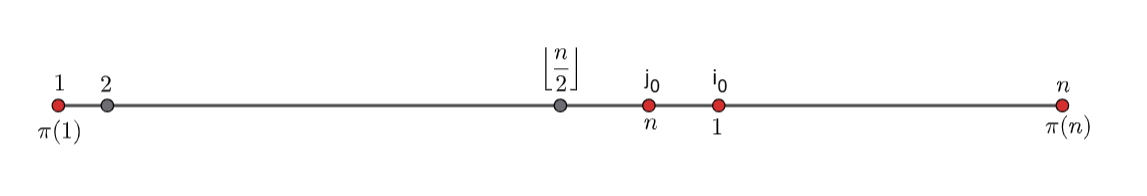}
        \caption{ $ i_0 > \displaystyle \left\lfloor\frac{n}{2}\right\rfloor$ }
    \end{figure} \\
    In this case,
    \begin{align*}
        c[1,j_0] = j_0 - (r-1)  & \geq      \displaystyle \left\lceil\frac{n}{2}\right\rceil - (r-1) \\
        & \geq  \left\lceil\frac{n}{2}\right\rceil - \left\lceil\frac{n}{4}\right\rceil + 1 \\
        & = \left\lfloor\frac{n+1}{4}\right\rfloor + 1.
    \end{align*} where the last step follows from property \ref{floor_ceil_5}. Since $\displaystyle \left\lfloor\frac{n + 1}{4}\right\rfloor + 1  = \left\lfloor\frac{5k+1}{4}\right\rfloor + 1 \geq k + 2$ for all integers $k \geq 3$, we are done in the case when $i_0 > \displaystyle \left\lfloor\frac{n}{2}\right\rfloor$. \\
   
    \textbf{Case 2: $i_0 \leq \displaystyle \left\lfloor\frac{n}{2}\right\rfloor$}. \\
    Consider $d = j_0 - i_0 \geq 1$. Since $\pi(i_0) =1$ and $\pi(j_0) = n$, 
    \begin{align}\label{eq: 1}
        c[i_0, j_0] = d + 1
    \end{align}
    Again, let $rank(1) = r$, so that 
    \begin{align}\label{eq: 2}
        c[1, j_0] = j_0 - (r-1),
    \end{align}
    as $rank(j_0) = j_0$.
    Also, 
    \begin{align}\label{eq: 3}
        c[1, i_0]  \geq r-(d-1),
    \end{align}
    Indeed, at most $(d-1)$ elements less than $\pi(1)$ can occur between $i_0$ and $j_0$ and these will not be included in the interval $[1, i_0]$. \\
     \begin{figure}[h]
        \centering
        \includegraphics[width=1\textwidth]{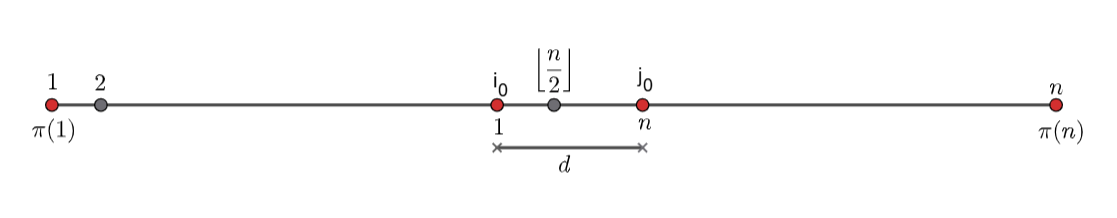}
        \caption{ $ \displaystyle i_0 \leq \left\lfloor\frac{n}{2}\right\rfloor$ }
        \label{fig 2}
    \end{figure} \\
    
    Define $rank_2(i)$ for $i \in [i_0, n]$ as
    \[
    rank_2(i) := \#  \left\{ k \in [i_0,n]\, \big| \, \pi(k) \leq \pi(i)  \right\}.
    \]
    and let $rank_2(n) = s$. Then
    \begin{align}\label{eq: 4}
        c[i_0,n] = s
    \end{align}
    and
    \begin{align}\label{eq: 5}
        c[j_0, n] \geq (n-j_0+1) - (s-2),
    \end{align}
    as at most $s-2$ elements smaller than $\pi(n)$ (excluding $1 = \pi(i_0)$ and $\pi(n)$ itself) can occur in the interval $[j_0, n]\setminus \{ j_0, n\}$ which will not count towards enumerating the elements of $c[j_0,n]$.\\
    
    Now adding equations \ref{eq: 1}, \ref{eq: 2} and \ref{eq: 4}, along with inequalities \ref{eq: 3} and \ref{eq: 5} we have,
    \begin{align*}
        c[i_0, j_0] +  c[1, j_0] +  c[1, i_0] + &  c[i_0,n] + c[j_0, n] \\
        & \geq  d + 1 + j_0 - (r-1) + r-(d-1) + s + (n-j_0+1) - (s-2) \\
        & = n + 6 = 5k + 6.\\
    \end{align*}
    Therefore, by the pigeonhole principle, there must be a caged sequence among the LHS of length at least $k + 2$. This establishes the lower bound  in the case $n=5k$. \\
    
    To complete the proof of the general statement of the theorem, we produce a permutation $\pi \in S_n$ with $n = 5k + 4$ for any positive integer $k$ that has caged length $c(\pi) = k +2$.\\
    
    Let us first introduce some further notation. For any distinct positive integers $m$ and $n$, $[m,n]$ shall denotes either of the following \emph{as a sequence}:
    \[
        [m,n] := 
        \begin{cases}
            (m,m+1,m+2,\dots,n), & \text{if $m < n$} \\
            (m, m-1, m-2, \dots, n), & \text{otherwise.}   
        \end{cases}  
    \]
Consider
    \[
    \pi_{ext}  := [3k+3, \, 4k+3]\,[k+1, \,1]\,[2k+3, \,3k+2]\,[5k+4, \,4k+4]\,[k+2,\,2k+2] 
    \]
    In the usual permutation form, $\pi_{ext}$ is given by 
    \[
 \begin{pmatrix}
    1 & \cdots & k+1 & k+2 & \cdots & 2k +2 & 2k +3 & \cdots & 3k + 2 & 3k + 3 & \cdots & 4k + 3 & 4k + 4 & \cdots & 5k + 4 \\
    3k + 3 & \cdots & 4k + 3 & k + 1 & \cdots & 1 & 2k + 3 & \cdots & 3k + 2 & 5k + 4 & \cdots & 4k + 4 & k + 2 & \cdots & 2k + 2 \\
  \end{pmatrix}
    \]
    To see why $c(\pi_{ext})  = k+2 $, partition the linear order form of the permutation as follows. 
    \[
    \pi_{ext}  = \underbrace{[3k+3, \, 4k+3]}_\text{$\bar{\mathbf{1}}$}\,\underbrace{[k+1, \,1]}_\text{$\bar{\mathbf{2}}$}\,\underbrace{[2k+3, \,3k+2]}_\text{$\bar{\mathbf{3}}$}\,\underbrace{[5k+4, \,4k+4]}_\text{$\bar{\mathbf{4}}$}\,\underbrace{[k+2,\,2k+2]}_\text{$\bar{\mathbf{5}}$}
    \] and let $\bar{\mathbf{i}}, i=1,\ldots,5$ denote the intervals as indicated above.\\
    
    By abuse of notation, by $c[\bar{\mathbf{i}}, \bar{\mathbf{j}}]$ indicate the length of any maximal caged sequence with endpoints in $\bar{\mathbf{i}}$ and $\bar{\mathbf{j}}$, respectively.  Observe that for any $i \in [5]$, $c[\bar{\mathbf{i}}, \bar{\mathbf{i}}] \leq k +1$ as the length of these intervals is at most $k + 1$ . Similarly, it is easy to see that $c[\bar{\mathbf{2}}, \bar{\mathbf{3}}] = |\bar{\mathbf{3}}| + 1 = k + 1$. By the same logic, $c[\bar{\mathbf{3}}, \bar{\mathbf{4}}] = k + 1 $ while $c[\bar{\mathbf{1}}, \bar{\mathbf{2}}] = c[\bar{\mathbf{1}}, \bar{\mathbf{4}}]= c[\bar{\mathbf{2}}, \bar{\mathbf{5}}] = c[\bar{\mathbf{4}}, \bar{\mathbf{5}}]= k + 2$. Finally, $c[\bar{\mathbf{1}}, \bar{\mathbf{5}}] = c[\bar{\mathbf{2}}, \bar{\mathbf{4}}] = |\bar{\mathbf{3}}| + 1 + 1 = k + 2$ since every element of the interval labeled $\bar{\mathbf{3}}$  is counted alongside at most one element from either of the terminal intervals ($\bar{\mathbf{1}}, \bar{\mathbf{5}}$ and $\bar{\mathbf{2}}, \bar{\mathbf{4}}$ respectively). Lastly, observe that $c[\bar{\mathbf{1}}, \bar{\mathbf{3}}], c[\bar{\mathbf{3}}, \bar{\mathbf{5}}]\le 2$. This establishes $c(\pi_{ext}) = k + 2$ and
 shows that for $n = 5k + 4$ we have $\min_{\pi \in S_n} c (\pi)\le k+2$. This, in conjunction with the lower bound established earlier for $n=5k$, together completes the proof of the more general statement of the theorem:
    \[
    \min_{\pi \in S_n} c (\pi) = \left\lfloor\frac{n}{5}\right\rfloor + 2.
    \]  
\end{proof}

\section{Proof of Theorem \ref{asymmetric} }

\begin{proof}
    Observe that $\pi^R$ defined earlier is a bijection from $S_n \longrightarrow S_n$ and flips a BCS into an FCS and vice versa. Consequently, it follows that $\mathcal{C}(k,l) = \mathcal{C}(l,k)$ for any positive integers $l,k \geq 3$, so it only remains to show that
    $$\mathcal{C}(k,l) =\mathcal{C}(l,k)=\min \{3k + 2l, 2k + 3l\} - 10.$$ 
    Without loss of generality, we may assume that $k \leq l$, so it suffices to show that 
    \[
     \mathcal{C}(k,l) = \mathcal{C}(\ell,k)  = 3k + 2\ell- 10.
    \]
    Let  $n = 3k + 2l - 10$ and $\pi \in S_n$. Again, let $\pi(i_0) = 1$ and $\pi(j_0) = n$. We have the following cases:
    \begin{enumerate}
   \item  \textbf{Case 1: $1<i_0 < j_0 < n$}. \\
    Observe that this case is similar to the case \ref{fig 2} in the proof of Theorem \ref{caged}. So, we borrow all the notation and definitions from there. Let $rank(1)=r$ and $rank_2(n)=s$. Also, let $d = j_0 - i_0 \geq 1$. Then all the equations and inequalities from \ref{eq: 1} to \ref{eq: 5} are valid. \\
    Now assume in contrast that the permutation $\pi$ does not contain a $k-$FCS and a $l-$BCS. Thus, we have
    \begin{align*}
        c[i_0, j_0] & \leq k-1, & c[1, j_0] & \leq k-1, \\
        c[1, i_0] & \leq l-1, &  c[i_0, n] & \leq k -1, \\
        c[j_0, n] & \leq l-1
    \end{align*} (see Fig. \ref{fig 2})
    Combining these inequalities with \ref{eq: 1} to \ref{eq: 5}, we get
    \begin{align*}
        3k + 2l -5 =  (k-1) +  (k-1) & + (l-1)  + (k-1) + (l-1) \\
         & \geq c[i_0, j_0] +  c[1, j_0] +  c[1, i_0] +   c[i_0,n] + c[j_0, n] \\
        & \geq  d + 1 + j_0 - (r-1) + r-(d-1) + s + (n-j_0+1) - (s-2) \\
        & = n + 6 \\
        & =  3k + 2l - 4,
    \end{align*} which is a contradiction, Therefore, $\pi$ either contains a $k-$FCS or an $\ell-$BCS.
    \item \textbf{Case 2: $1=i_0 < j_0 < n$}:
     In this case, if $c[1, j_0] = j_0 \geq k $, we are done. Thus, assume that $j_0 \leq k-1$. Redefine $rank_2(i)$ for all $i \in [j_0,n] $ as follows:
     \[
       rank_2(i)  := \#  \left\{ k \in [j_0,n]\, \big| \, \pi(k) \leq \pi(i)  \right\} 
     \]
     Let $rank_2(n) = s$. Note that there are $n-j_0+1 \geq 3k + 2l-10 - (k-1) + 1 = 2k + 2l - 8$ elements in $[j_0,n]$. Therefore, if $s \geq k + l -4$ $\left(=\displaystyle \frac{1}{2}(2k + 2l -8)\right)$, 
     $$c[1,n] \geq s + 1 \geq k + l -3 \geq k,$$
     as $l \geq 3$. 
     On the other hand, if $s < k + l -4$, 
     \[
      c[j_0, n] = (n-j_0+1) - (s-1) > 2k + 2l - 8 - (k + l - 4 - 1) = k + l - 3 \geq l,
     \] as $k \geq 3$. \\
     (Observe that in the definition of $rank_2(i)$, $\pi(1)=1$ is not counted.) \\
     Therefore, we are done in both cases.     
\item     \textbf{Case 3: $1<i_0<j_0=n$}. \\
     This case is similar to the previous case. Redefining $rank(i)$ for $i \in [1, i_0]$ and following similar arguments, we can easily get the desired result. Therefore, we omit the proof here. \\
     Next, we will go for the cases where $j_0< i_0$. 
\item    \textbf{Case 4: $1=j_0 < i_0 < n$}. \\
     If $c[1, i_0] = i_0 \geq l$, we are done. Therefore, assume $i_0 \leq l-1$. Now, define $rank_2(i)$ for all $i \in [i_0,n]$ as follows:
     \[
     rank_2(i)  := \#  \left\{ k \in [i_0,n]\, \big| \, \pi(k) \leq \pi(i)  \right\} 
     \]
     Let $rank_2(n) = s$. Since there are $n-i_0+1 \geq 3k + 2l -10 - (l-1) + 1 = 3k + l - 8$ elements in $[i_0, n]$, there exists a $i$ with $rank_2(i) = k$ (as $k < 3k + l - 8 \iff 2k + l  > 8$ which is true as $k, l \geq 3$). Thus, if $rank_2(n) = s \geq k$, $c[i_0,n] = s \geq k $, so we get a $k-$FCS. Also, if $s < k$, 
     \[
     c[1, n] \geq (n - i_0+1) - (s-1) + 1 > 3k + l - 8 - k + 2 = 2k + l - 6 \geq l,
     \]
    as $k \geq 3$, so we get a $(l+1)-$BCS.
\item     \textbf{Case 5: $1<j_0<i_0=n$}. \\
    This case can be proved in a similar way as the last case by redefining $rank(i)$ appropriately. Therefore, we omit its proof. 
\item    \textbf{Case 6: $1<j_0<i_0<n$}.     Let $\pi(1) = a_1$ and $\pi(n) = a_n$. \\

    \textbf{Case 6(a): $a_1>a_n$}. \\
    It is similar to the first case of this proof, except that the order of $i_0$ and $j_0$ is reversed. Thus, we will redefine the functions $rank(i)$ and $rank_2(i)$ according to the need:
    \begin{align*}
        rank(i) & := \#  \left\{ k \in [1,i_0]\, \big| \, \pi(k) \leq \pi(i)  \right\},  \\
         rank_2(i) & := \#  \left\{ k \in [j_0,n]\, \big| \, \pi(k) \leq \pi(i)  \right\} 
    \end{align*}
    Assume $rank(1) = r$ and $rank_2(n) = s$. Define the sets $X, \, Y$ and $Z$ as follows:
    \begin{align*}
        X & :=\{ \pi(i) \vert \, 1\leq i \leq j_0-1 \}, \\
        Y & := \{ \pi(i) \vert \, j_0\leq i \leq i_0 \}, \\
        Z & := \{ \pi(i) \vert \, i_0+ 1\leq i \leq n \}.
    \end{align*}
    Also, let
    \begin{align*}
        r_1 &  = \#\{ x < a_1 \, \vert \, x \in X \}, &
        r_2 & = \#\{ y < a_1 \, \vert \, y \in Y \} ,\\
        s_1 & = \#\{ z < a_n \, \vert \, z \in Z \}, &
        s_2  &  = \#\{ y < a_n \, \vert \, y \in Y \}.
    \end{align*}
    Clearly, $r-1 = r_1 + r_2$ and $s-1 = s_1+s_2$. Now, following earlier arguments (see \ref{eq: 1} to \ref{eq: 5}), we get  
    \begin{align*}
        c[j_0, i_0] & = d + 1  ,    &        c[1, j_0] & = j_0 - r_1, \\
         c[1, i_0] &  = r , &         c[i_0,n] & = s_1 + 2, \\
         c[j_0, n] & = (n-j_0+1) - (s-1).
    \end{align*}
    If $\pi$ does not contain a $k-$FCS and a $l-$BCS, then it follows that
    \begin{align*}
        c[j_0, i_0] & \leq l-1, & c[1, j_0] & \leq k-1, \\
        c[1, i_0] & \leq l-1, &  c[i_0, n] & \leq k -1, \\
        c[j_0, n] & \leq l-1.
    \end{align*}
    Now, combining the above two sets on inequalities,
    \begin{align*}
       & n-j_0+1 - s + 1  \leq l-1 \hspace{5mm} & (\text{since } n-j_0+ 1-(s-1) = c[j_0, n] \leq l-1) \\
       \implies & 3k + 2l - 8 - j_0 - s  \leq l - 1 & (\text{since } n = 3k + 2l - 10) \\
       \implies & 3k + l  \leq j_0 + s + 7 \\
       \implies & 3k + r + 1 \leq j_0 + s + 7 & (\text{since } r = c[1, i_0] \leq l -1 ) \\
       \implies & 3k + r_1 + r_2  \leq j_0 + s_1 + s_2 + 6 & (\text{since } r-1 = r_1 + r_2, \, s-1 = s_1 + s_2) \\
       \implies & 2k + j_0 + r_2 + 1 \leq j_0 + s_1 + s_2 + 6 &(\text{since } j_0 - r_1 = c[1, j_0] \leq k -1) \\
       \implies & 2k + r_2 \leq s_1 + s_ 2 + 5 \\
       \implies & 2k + r_2 \leq k + s_2 + 2  & (\text{since } s_1 + 2 = c[i_0, n] \leq k -1) \\
       \implies & r_2 - s_2 \leq 2 - k.
    \end{align*}
    Since $a_1>a_n$, by definition, $r_2 \geq s_2$ which implies $k\leq 2$, a contradiction. Therefore, this case must also give either a $k-$FCS or a $l-$BCS. \\    \textbf{Case 6(b): $a_1<a_n$}. \\
    Note that for any $1 \leq i<j\leq n$, 
    \[
    c[i, j] = \sum_{k:i<k<j} \mathbb{1}_{\{\pi(i) < \pi(k) < \pi(j) \text{ or } \pi(i) > \pi(k) > \pi(j)\}} + 2.
    \]
    Take any permutation $\pi$. Let $\sigma = \pi^{-1} $ (the usual group theoretic inverse). If for any $i<j<k$, $\pi(i)<\pi(j)<\pi(k)$, then in $\sigma$, we have $\sigma(\pi(i))<\sigma(\pi(j)) < \sigma(\pi(k))$ for $\pi(i)<\pi(j)<\pi(k)$. Similarly, if for any $i<j<k$, $\pi(i)>\pi(j)>\pi(k)$, then in $\sigma$, we have $\sigma(\pi(k)) >\sigma(\pi(j)) > \sigma(\pi(i)) $ for $\pi(k)<\pi(j)<\pi(i)$. In simple terms, this means that for every FCS and BCS in $\pi$, there is an equivalent FCS and BCS in $\sigma$ respectively. \\
    Now, consider any permutation $\pi$ corresponding to the given case. Let $\sigma = \pi^{-1} $. Then we have $\sigma(a_1)=1$ and $\sigma(a_n)=n$. Thus, for $\sigma$, $i_0(\sigma) = a_1$ and $j_0(\sigma) = a_n$. By our assumption, we have $1<i_0(\sigma)<j_0(\sigma)<n$ which is exactly the first case of this proof. Thus, the result holds for $\sigma$ and hence for $\pi$ as well. This finishes the proof for the sixth and last case. \\
    So far, we have proved that for any $k, l \geq 3$,
    \[
    \mathcal{C}(k,l) = \mathcal{C}(l,k)  \leq  \min \{ 3k + 2l, 2k + 3l\} - 10.
    \]
    To obtain the equality, we must give a permutation $\pi(k,l) \in S_{n-1}$ that contains neither a $k-$FCS nor a $l-$BCS. Again, without loss of generality, we assume $k\leq l$ and thus take $n = 3k + 2l -10$. Using the earlier notation of $[m,n]$, we define the permutation as 
    \[
    \pi(k,l) := \underbrace{[2k+l-6,\, 3k+l-9]}_\text{$\bar{\mathbf{1}}$} \underbrace{[l-2, \, 1]}_\text{$\bar{\mathbf{2}}$}\underbrace{[k+l-3, \, 2k + l - 7]}_\text{$\bar{\mathbf{3}}$}\underbrace{[3k + 2l -11, \, 3k + l - 8]}_\text{$\bar{\mathbf{4}}$}\underbrace{[l-1, \, k + l -4]}_\text{$\bar{\mathbf{5}}$}.
    \]
    Recall the earlier notation $c[\bar{\mathbf{i}}, \bar{\mathbf{j}}]$. For $i = 1,3,5$, the FCSs $c[\bar{\mathbf{i}}, \bar{\mathbf{i}}] \leq k -2 $ as the length of any of these intervals is at most $k -2$. Similarly, for $i=2,4$, the BCSs $ c[\bar{\mathbf{i}}, \bar{\mathbf{i}}] = l -2  $. Other major BCSs $c[\bar{\mathbf{1}}, \bar{\mathbf{2}}]$ and $c[\bar{\mathbf{4}}, \bar{\mathbf{5}}]$ are of length $l-1$. The BCS $c[\bar{\mathbf{1}}, \bar{\mathbf{5}}] $ is of length $k-1 \leq l-1$ (note that $|\bar{\mathbf{3}}|=k-3$). Similarly, the major FCSs $c[\bar{\mathbf{1}}, \bar{\mathbf{4}}] $, $c[\bar{\mathbf{2}}, \bar{\mathbf{4}}] $, and $c[\bar{\mathbf{2}}, \bar{\mathbf{5}}] $ are of length $k-1$. The rest can be argued in the same way. Ultimately, there is neither a $k-$FCS nor a $l-$BCS. \\
    This shows that the inequality proved earlier is indeed an equality, and hence establishes the statement of the theorem: for any $k,l \geq 3,$ 
    \[
     \mathcal{C}(k,l) = \mathcal{C}(l,k)  = \min \{ 3k + 2l, 2k + 3l\} - 10.
    \]\end{enumerate}
\end{proof}
\subsection{Caged length vs number of runs: some remarks}
\begin{definition}[Peaks and Valleys]
    Given a permutation $\pi=(\pi(1), \pi(2), \dots, \pi(n))$, an element $\pi(i)$ (for $2\le i\le n-1$) is called a peak if $\pi(i-1)<\pi(i)>\pi(i+1) $ (a local maxima). Similarly, an element $\pi(i)$ is called a valley if $\pi(i-1)>\pi(i)<\pi(i+1) $ (a local minima).  
\end{definition}
$\pi(1)$ and $\pi(n)$ are not classified as peaks or valleys. 
\begin{definition}[Runs in a permutation]
    The contiguous increasing or decreasing segments between two successive peaks and valleys are called runs.
\end{definition}
In our earlier example of permutation $(8,9,5,6,14,13,2,16,15,1,4,3,11,12,7,10)$ of $[16]$, $9, \, 14, \, 16, \, 4$ and $ 12  $ are peaks, while $5, \, 2, \, 1, \, 3 $ and $7$ are valleys. Blocks like $(5, 6, 14)$ and $(3, 11, 12)$ are increasing runs and those like $(14, 13, 2)$ and $(16, 15, 1)$ are decreasing runs. It is easy to see that 
\[
\#\text{runs in a permutation} = \#peaks \, + \, \#valleys + 1.
\]
Note that the extremal examples of permutations that we have seen so far, $\pi_{ext}$ and $\pi(k,l)$, contained exactly five runs. So, it naturally begs the following question: If we change the number of runs, do the bounds still hold or do we get different lengths? \\

Let $S_n^k$ denote the set of all permutations of $[n]$ with exactly $k$ runs. Clearly, if there are exactly $k$ runs in $\pi$, then at least one run has length at least $n/k$, then $c(\pi)\ge  n/k$ and for for $1\le k\le 5$, this is clearly also best possible.  For $k\ge 5$, it is apriori possible that an increase in the number of runs might result in the presence of larger caged sequences. Perhaps surprisingly, this is not the case.
\begin{proposition}
    For $k \geq 5$, 
    \[
    \min_{\pi \in S_n^k} c(\pi) \sim \frac{n}{5}.
    \]
    In fact, if $\pi$ is an `zig-zag' permutation, that is, if $k=n-1$, we have
    \[
    \min_{\pi \in S_n^{n-1}} c(\pi) \leq \left\lfloor \frac{n}{5} \right\rfloor + 6.
    \]
\end{proposition}
\begin{proof}
    Since by Theorem \ref{caged}, $\displaystyle\min_{\pi \in S_n^k} c(\pi) \geq \left\lfloor\frac{n}{5}\right\rfloor + 2$, it is enough to get an example of a permutation with $k$ runs that has the caged length of order $\displaystyle\frac{n}{5}$. To produce such an example, let us take our extremal example with $5$ runs, $\pi_{ext} $. For $m \in \mathbb{N}$,
    \[
    \pi_{ext}  = \underbrace{[3m+3, \, 4m+3]}_\text{$\bar{\mathbf{1}}$}\,\underbrace{[m+1, \,1]}_\text{$\bar{\mathbf{2}}$}\,\underbrace{[2m+3, \,3m+2]}_\text{$\bar{\mathbf{3}}$}\,\underbrace{[5m+4, \,4m+4]}_\text{$\bar{\mathbf{4}}$}\,\underbrace{[m+2,\,2m+2]}_\text{$\bar{\mathbf{5}}$}.
    \]
    If $k = 5$, the example works. If $k > 5$, do the following procedure: \\
    If $k = 6$, take the interval $\bar{\mathbf{1}}$ and perturb it a little from the left end to make an extra run. For example, for $m = 7$, the interval will be $(24, 25, \dots, 31)$, which constitutes a single run. Change it to $(25, 24, 26, \dots, 31)$ to add a valley at $24$. Now you have $6$ runs in total. Similarly, if $k = 7$, move the first $3$ elements, $(24, 25, 26, 27, \dots, 31) \longrightarrow (24, 26, 25, 27, \dots , 31)$, to get $2$ extra runs. \\
    This procedure easily extends to other intervals as well. Note that to accommodate the maximum possible runs, the lengths of the intervals $\bar{\mathbf{2}}$ and $\bar{\mathbf{4}}$ must be odd and that of $\bar{\mathbf{3}}$ must be even. This can be managed by changing the boundaries of the intervals with the intervals $\bar{\mathbf{1}}$ and $\bar{\mathbf{5}}$, if needed. \\
    Observe that the above procedure creates any $k \geq 5$ number of runs using local inversions, while globally maintaining the monotonicity of each interval (see Figure \ref{fig:D-Imager}). These local inversions boost the caged length by at most $2$ as they add an additional element from both ends. Another extra $+2$ comes due to boundary changes, described in the end of the procedure. 
    \begin{figure}
\centering
    \begin{subfigure}[b]{0.5\textwidth}            
            \includegraphics[width=\textwidth]{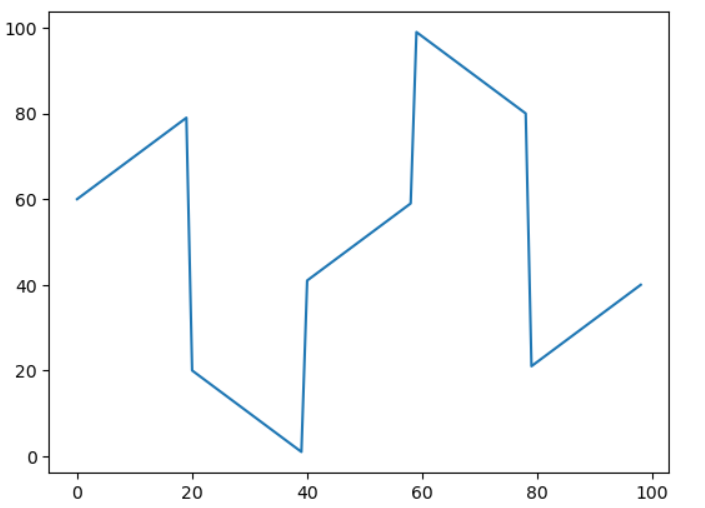}
            \caption{Permutation $\pi_{ext}$,  $k=5$. }
            \label{fig:SRl}
    \end{subfigure}%
    \begin{subfigure}[b]{0.5\textwidth}
            \centering
            \includegraphics[width=\textwidth]{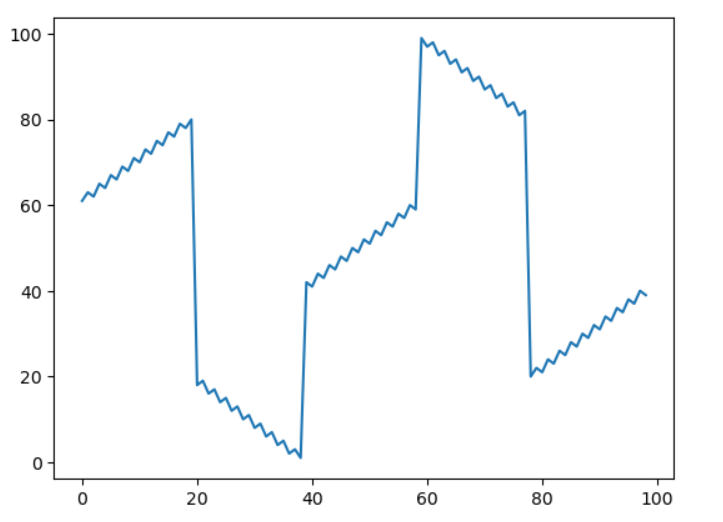}
            \caption{Zig-zag permutation, $k = 98$}
            \label{fig:D-Imager}
    \end{subfigure}
    \caption{Extremal permutations for $n = 99$}\label{fig:ext}
\end{figure}
\end{proof}
\section{$c(\pi)$ for random $\pi$: Proofs of Theorems \ref{expectation} and \ref{sqaure}}
\begin{proof}[Proof of theorem\ref{expectation}]
 Note that a uniformly random permutation can be generated by choosing the $X_i \sim U[0,1]$ i.i.d. for $1 \leq i \leq n$. So, we can work on these $X_i$'s instead. A long caged sequence will appear if two extreme $X_i$'s(say $X_{i_1}$ and $X_{i_2} $) assume the extreme values in the interval $[0, 1]$ and many other $X_j$'s whose indices fall between $i_1$ and $i_2$ take the values in between. We will now implement this idea mathematically. \\
 Without loss of generality, let $\omega(n)$ be any small function $\left(o(\sqrt{n})\right)$ such that $\omega(n) \xrightarrow{n \rightarrow \infty } \infty$. Take $\displaystyle\epsilon = \frac{\omega(n)}{\sqrt{n}}$. Define $\delta := \displaystyle \frac{1 - (1 - \epsilon)^{\frac{1}{3}}}{2} $. Observe that $0<\delta < < 1$ for $n$ is large. Let $\mathcal{E}_1 $ be the event that for some $i \in [1, \delta n]$, $X_i \in [0, \delta ]$. Similarly, let $\mathcal{E}_2 $ be the event that for some $j \in [(1 -\delta) n, n]$, $X_j \in [(1 - \delta), 1]$. Then 
 \[
 \mathbb{P}\left( \mathcal{E}_1 \wedge \mathcal{E}_2 \right) = 1 - \mathbb{P}\left( \bar{\mathcal{E}}_1 \vee \bar{\mathcal{E}}_2 \right)
 \] 
 Since the random variables $X_i$ are independent, $ \mathbb{P}\left( \bar{\mathcal{E}}_1 \right) = \mathbb{P} \left( \bar{\mathcal{E}}_2 \right) = (1 - \delta )^{\delta n} $. By union bound,
 \[
 \mathbb{P}\left( \bar{\mathcal{E}}_1 \vee \bar{\mathcal{E}}_2 \right) \leq 2(1 - \delta )^{\delta n} \leq 2e^{-\delta^2 n}.
 \] 
 So, 
 \[
 \mathbb{P}\left( \mathcal{E}_1 \wedge \mathcal{E}_2 \right) \geq 1 -  2e^{-\delta^2 n}.
 \]
 Define $Y : = \displaystyle \sum_{k \in (\delta n, (1 - \delta)n)} \displaystyle  \mathbb{1}_{X_k \in (\delta , (1- \delta ))} $. Observe that $Y$ corresponds to the length of the caged sequence. Using linearity of expectation, $\mu := \mathbb{E}( Y) = (1 - 2\delta)^2 n$. 
 Let $\mathcal{E}_3$ be the event that $Y$ is at least $(1 - \gamma) \mu$ for some small $\gamma >0$. Using the Chernoff bound for the sum of Bernoulli i.i.ds,
 \[
 \mathbb{P}\left( \bar{\mathcal{E}}_3 \right) = \mathbb{P}\left( Y \leq (1 - \gamma)\mu  \right) \leq \displaystyle e^{-\gamma^2 \mu / 2}.
 \]
 Finally, we use independence to get  
 \[
 \mathbb{P}\left( \mathcal{E}_1 \wedge \mathcal{E}_2 \wedge \mathcal{E}_3  \right) \geq \left( 1 -  2e^{-\delta^2 n} \right) \left( 1 - e^{-\gamma^2 \mu / 2} \right).
 \]
 Put $\gamma = 2 \delta$. Then the above statement implies $ \mathbb{P}\left(c(\pi) \geq (1 - 2 \delta)^3 n \right) \longrightarrow 1$ as $ n \rightarrow \infty$. \\
 By plugging the values of $\delta$ and $ \epsilon $, $\mathbb{P}\left( c(\pi) \geq \left( 1- \frac{\omega(n)}{\sqrt{n}} \right)n \right) \longrightarrow 1 $ as $n$ tends to infinity. \\
 To prove the second statement of the theorem, we need the following lemma:
\begin{lemma}\label{lemma 4.2}
   Let \(X,Y\sim U[0,1]\) be i.i.d. Let \(D = |X-Y|\) and \(N\mid D \sim \text{Bin}(n,D)\). Then for any $d = d(n)$ where $d=o(n)$, we have
   $$ \mathbb{P}(N>n-d(n)) \leq \exp{\left(-\frac{9}{8}\cdot d(n)\right)} + 16\left(\frac{d(n)}{n}\right)^2.$$
\end{lemma}
\begin{proof}
Set $t = n - d(n)$ with $d(n) = o(n)$. Write $M = n-N, Q= 1-D$, and define $M$ conditioned on $D$ to be  $\text{Bin}\big(n,Q\big)$. It is straightforward to check that $Q$ has  density $f_Q(q) = 2q$ for $q \in [0, 1]$.\\
Now,
\begin{align*}
    \mathbb{P}(N>n-d(n)) & = \mathbb{P}\left( M < d(n) \right) \\
   &= \mathbb{E}_Q\left[ \mathbb{P}(M<d(n) ) \mid Q\right] \\
   & = \int_0^1 \mathbb{P}\big(M<d(n)\mid Q=q\big)\ 2q\ dq.
\end{align*}
 For $0\le q\le q_0:=4d(n)/n$, we have $\mathbb{P}(M<d(n)\mid q)\le 1$, so
    \begin{align}\label{eq:9}
        \int_0^{q_0} \mathbb{P}(M<d(n)\mid q)\ 2q\  dq
    \le \int_0^{q_0} 2q\ dq
    = 16\left(\frac{d(n)}{n}\right)^2.
    \end{align}
    For $q_0 < q \leq 1$, the mean of $M$ is $nq > 4d(n)$, and as $M<d(n)$ is a lower tail event, a Chernoff-type bound gives 
    \begin{align*}
         \mathbb{P}(M<d(n)\mid q) & = \mathbb{P}\left(  M < \mathbb{E}(M)  -(nq - d(n))\right) \\
        & \leq \exp{\left( -\frac{(nq - d(n))^2}{2nq} \right)} \\
       & \leq \exp{\left( -\frac{9}{8}\cdot d(n) \right)},
    \end{align*}
    for all $q_0 < q \leq 1$. Thus,
    \begin{align}
        \int_{q_0}^1 \mathbb{P}(M<d(n)\mid q)\ 2q\  dq \le \exp\left(-\frac{9}{8}\cdot d(n)\right).
    \end{align}
Thus, we have 
$$
\mathbb{P}(N>n-d(n)) \leq 16\left(\frac{d(n)}{n}\right)^2 + \exp\left(-\frac{9}{8}\cdot d(n)\right).
$$
\end{proof}
We are now in a position to prove the second statement of theorem \ref{expectation}. \\
Set $d(n)=o(n)$. We want to estimate the probability that there exists any caged sequence of length $n - d(n)$. If this holds, there must exist a pair $i, j $ with $j - i  \geq n - d(n) - 1$ and $c[i, j] \geq n - d(n)$. Note that there are at most $(d(n))^2$ pairs with $j - i \geq n - d(n) - 1$. Pick any such pair $(i, j)$. In terms of earlier notation, we have the following.
\[
   c[i, j] = 2 \, +  \displaystyle \sum_{i < k < j} \displaystyle  \mathbb{1}_{X_i < X_k <X_j \text{ or } X_i > X_k > X_j}. 
\]
Write $N = c[i, j] - 2$. Clearly, $N \mid D \sim \text{Bin}(j - i - 1, D)$ where $D = |X_i - X_j|$. Now, $j - i - 1 \geq n - d(n) - 2 = (1-o(1))n$. Applying the lemma \ref{lemma 4.2}, 
\[
  \mathbb{P}\left( c[i, j] \geq n - d(n) \right) \leq  \exp{\left(-\frac{9}{8}\cdot (d(n) + 2) \right)} + 16\left(\frac{d(n) + 2}{(1+o(1))n}\right)^2.
\]
Using the union bound, 
\[
\mathbb{P}\left( \text{ There exists such a pair $i, j$ with } c[i, j] \geq n - d(n) \right) \leq \frac{d(n)^2}{\exp{\left(\frac{9}{8}\cdot (d(n) + 2) \right)}} + 16\frac{(d(n) + 2)^4}{(1+o(1))n^2}.
\]
For $d(n) = \frac{\sqrt{n}}{\omega(n)}$ with $\omega(n) = o(\sqrt{n})$ and $\omega(n) \rightarrow \infty $, the RHS in the above expression tends to $0$; this completes the proof.
\end{proof}

\begin{proof}[Proof of theorem \ref{sqaure}]
We shall describe the proof in the square shape case  ($\theta=1$) in full detail; the case for other $\theta$ is similar, and so we shall not go over all the relevant details there.\\

First, observe that to get a `super-long' caged sequence, there must occur some $i$ in a `small' proportion of initial indices that satisfies that $\pi(i)$ is either significantly small or significantly large. We will make this more precise momentarily.

Suppose $\pi \in \mathrm{ES}_n$ is picked  uniformly at random. Consider the set $[1, \delta N]$ for some fixed $0<\delta < 1/2$. By Theorem \ref{sandwich}, the plot of $\pi$ is given in terms of the tableaux $P$ and $Q$ (which are uniformly independent random $n \times n$ square tableaux) by 
\[
A_\pi = \left\{ (q_{i,j}, \, p_{n+1-i, j}) \, \middle\vert \, 1 \leq i, j \leq n \right\}.
\]
By Theorem \ref{limit_shape}, asymptotically, with high probability, each scaled point $\displaystyle \frac{1}{n^2}(q_{i,j}, \, p_{n+1-i, j})$ is uniformly close to the point $\left(L(x, y), \, L(1-x, y)\right)$, where $x = \frac{i}{n}, \, y = \frac{j}{n}$. We are interested in the points with $q_{i, j} \leq \delta N $, which when scaled, corresponds to those points for which $L(x, y) \leq \delta$ (See figure \ref{fig:delta}). In other words, we want to find the level curve ($L(x, y) = \alpha$) with the smallest possible $\alpha$, that non-trivially meets the region $L(1-x, y) \leq \delta$ (figure \ref{fig:delta 2}). \\
In the rotated $u-v$ system, the $\delta-$level curve has parameterization $\left(u, g_{\delta}(u)\right) $. Thus, in the $x-y$ system, it has the parametrization
\[
\left( \frac{u+g_\delta(u)}{\sqrt{2}}, \frac{g_\delta(u)-u}{\sqrt{2}} \right),
\]where $|u| \leq \sqrt{2\delta(1-\delta)}$. The corresponding reflected curve $L(1-x, y) = \delta$ is thus parametrized as 
\[
\left( 1-\frac{u+g_\delta(u)}{\sqrt{2}}, \frac{g_\delta(u)-u}{\sqrt{2}} \right).
\] We are interested in the point on this curve that lies on the smallest $\alpha-$level curve for some $\alpha>0$. Clearly, this point corresponds to the parameter $u = \sqrt{2\delta(1-\delta)}$ and therefore has coordinates $\left( 1- 2\sqrt{\delta(1-\delta)}, \,  0 \right)$ (see figure \ref{fig:delta 2}). As a point on the $\alpha-$level curve, its coordinates are $\left(  2\sqrt{\alpha(1-\alpha)}, \,  0 \right)$, which forces the relation 
\[
    2\sqrt{\alpha(1-\alpha)} + 2\sqrt{\delta(1-\delta)} = 1.
\] 
\begin{figure}
\centering
    \begin{subfigure}[b]{0.35\textwidth}            
            \includegraphics[width=\textwidth]{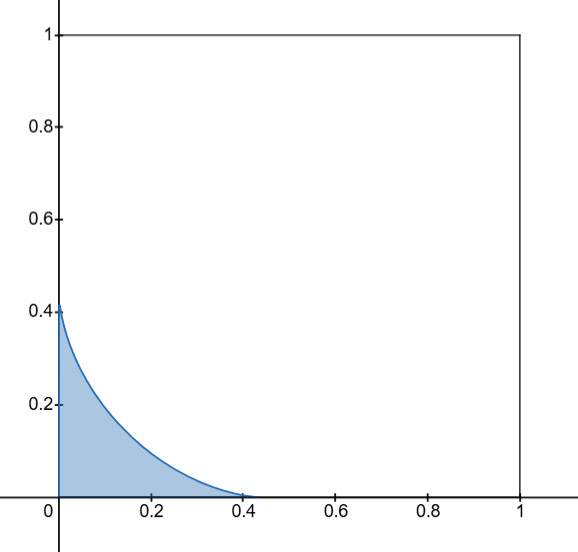}
            \caption{ $q_{i, j} \leq \delta N$}
            \label{fig:delta}
    \end{subfigure}%
    \hspace{7mm}
    \begin{subfigure}[b]{0.35\textwidth}
            \centering
            \includegraphics[width=\textwidth]{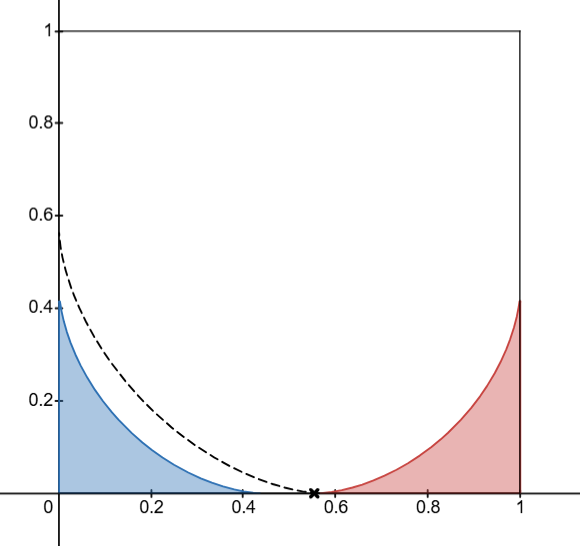}
            \caption{$\min p_{n + 1-i, j}  $ for $q_{i, j} \leq \delta N $}
            \label{fig:delta 2}
    \end{subfigure}
    \caption{Minimum value for $\delta-$proportion of initial indices}\label{fig:shaded}
\end{figure}
A simple calculation will imply that $\alpha \rightarrow 1/2$ as $\delta \rightarrow 0$. Similarly, one can argue about the $\alpha-$level curve for maximum $\alpha$ and get the same result. 
This shows that for sufficiently small $\delta$, the initial $\delta N$ indices do not assume significantly ``small" or ``large" values with probability tending to $1$ (as $n \rightarrow \infty$), and therefore the first part of theorem \ref{sqaure} holds. \\


Before we launch into the proof of the second part, observe that for any permutation $\pi\in S_n$ and for $i<j$, $c[i,j]$ simply counts the number of points contained in the axis-parallel rectangle with $(i,\pi(i))$ and $(j,\pi(j))$ as opposite vertices.

Let $\mathcal{Z}, L(s,t)$ be as in the statement of theorem \ref{shape_square} and let $\varphi:[0,1] \times[0,1] \rightarrow \mathcal{Z}$ be the 1-1 and onto mapping defined by
$$
\varphi(s, t)=\left(2 L(s, t)-1,2 L(1-s, t)-1\right).
$$

By the previous observations and by the content of theorem \ref{shape_square}, it suffices to find an axis-parallel rectangle contained in the region $\mathcal{Z}$ of maximum measure w.r.t. the density induced by the pull-forward map $\phi$. By the second statement of theorem \ref{shape_square}, it follows that we may assume that the vertices of such a rectangle $R$ lie on the boundary of $\mathcal{Z}$ (figure \ref{fig:rectangle}).
\begin{figure}
\centering
    \begin{subfigure}[b]{0.35\textwidth}            
            \includegraphics[width=\textwidth]{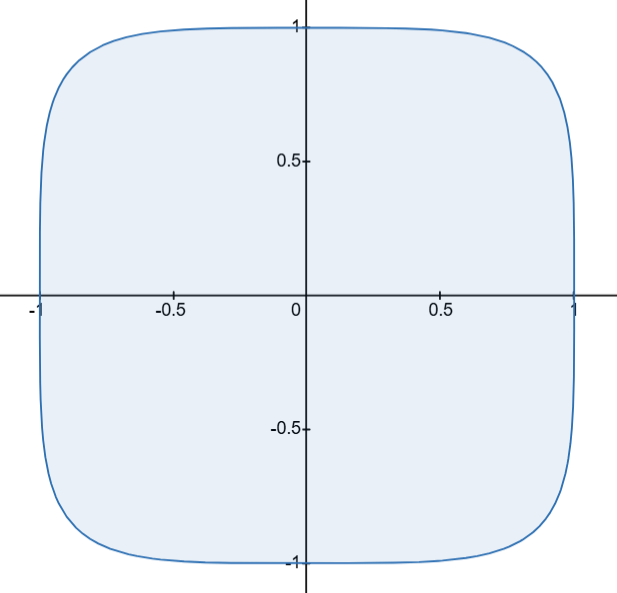}
            \caption{ The limiting shape $\mathcal{Z}$}
            \label{fig:Z}
    \end{subfigure}%
    \hspace{7mm}
    \begin{subfigure}[b]{0.35\textwidth}
            \centering
            \includegraphics[width=\textwidth]{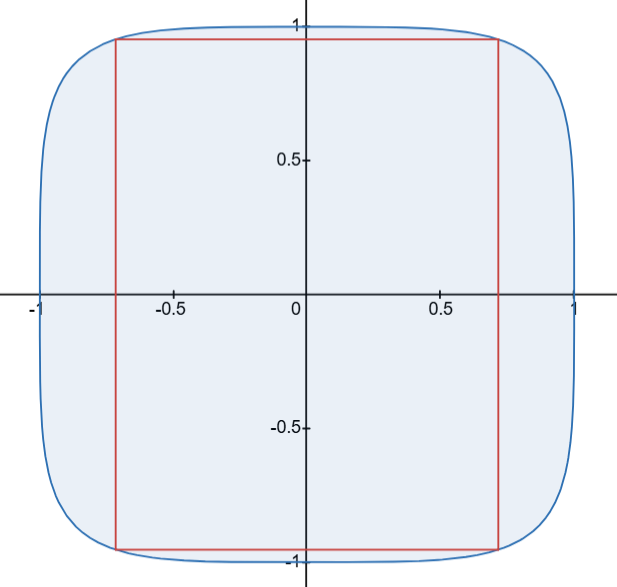}
            \caption{Axis-parallel rectangle $R$ in $\mathcal{Z}$}
            \label{fig:rectangle}
    \end{subfigure}
    \caption{Limiting shape $\mathcal{Z}$ of a random $ES$ permutation (Rectangle corresponds to some maximal caged sequence).}
\end{figure}
So, it is enough to find an axis-parallel rectangle $R$ that maximizes the integral $\int_R| J_{\varphi^{-1}}(x, y)|d x d y$, where $J_{\varphi^{-1}}$ is the Jacobian of the mapping $\varphi^{-1}$. \\
To compute this, it is simpler to pull back the image of such a rectangle $R$ and compute the area of $\varphi^{-1}(R)$ (see figure \ref{figure 7}). Moreover, since the rectangle is axis-parallel and the boundary curve of $\mathcal{Z} $ is symmetric in $X, Y$, the vertices of $R$ must be of the form $r_1=(a, b)$, $r_2=(-a, b)$, $r_3 = (-a, -b)$ and $r_4=(a, -b)$ with
\begin{align}\label{eq:7}
    \left(a^2-b^2\right)^2+2\left(a^2+b^2\right) = 3. 
\end{align}
  We have explicit formulae for $L(x,y)$ on the boundary \cite{lim_shapes}:
\begin{align*}
   & L(t, 0) = L(0, t) = \frac{1-\sqrt{1-t^2}}{2}, \\
   & L(t, 1)= L(1, t) = \frac{1+\sqrt{2t - t^2}}{2}.
\end{align*}
Plugging them in $\varphi$ and following some simple calculations, we get $(t, 0) \xmapsto{\varphi} \left(-\sqrt{1-t^2}, -\sqrt{2t - t^2}\right) $. In other words, the side of the square $[0, 1] \times [0,1]$ along the positive $x-$axis gets mapped to the part of the boundary curve of $\mathcal{Z}$ that lies in the third quadrant (see figure \ref{figure 7}), and the other sides are mapped in an orientation-preserving manner. Thus, the vertex $r_1=(a, b)$ (assuming both $a$ and $b$ are positive) corresponds to a point of the form $(t, 1)$ in its pre-image. \\
Now, the rectangle $R$ in $\mathcal{Z}$ is defined by its sides $y = \pm b$ and $x = \pm a$. Via the inverse image of $\varphi$, these give the level curves $\displaystyle L(x, y) = \frac{1\pm a}{2}$ and $\displaystyle L(1-x, y) = \frac{1\pm b}{2}$ as their pre-images. 
\begin{figure}
\centering
    \begin{subfigure}[b]{0.38\textwidth}            
            \includegraphics[width=\textwidth]{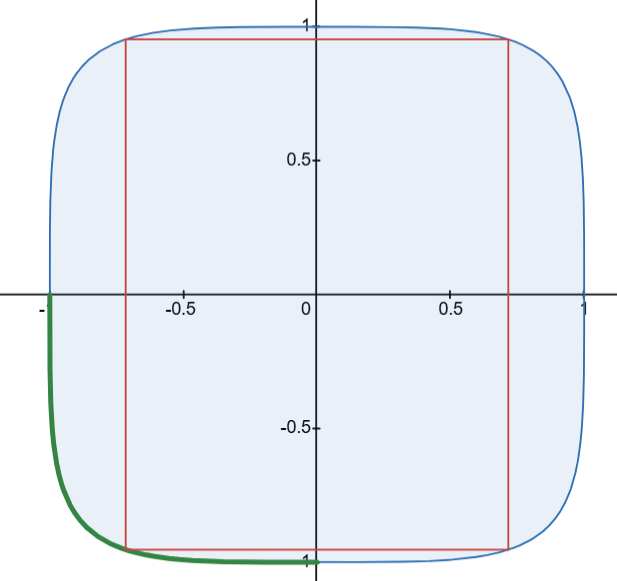}
            \caption{ The rectangle $R$ in $\mathcal{Z}$}
            \label{fig:rect_R}
    \end{subfigure}%
    \hspace{7mm}
    \begin{subfigure}[b]{0.35\textwidth}
            \centering
            \includegraphics[width=\textwidth]{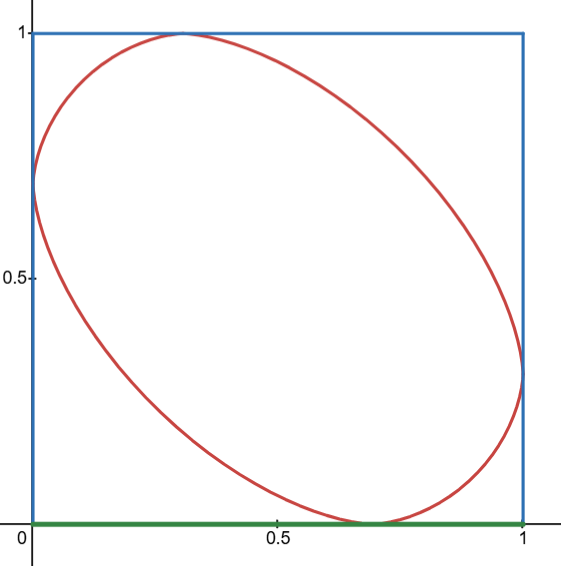}
            \caption{Pre-image of $R$ in $[0, 1] \times [0, 1]$}
            \label{fig:pre_R}
    \end{subfigure}
    \caption{Visual depiction of $\varphi^{-1}$ over an axis-parallel rectangle (the marked(green) part is mapping of one part of the boundary, all other parts are mapped similarly preserving the orientation).}\label{figure 7}
\end{figure}
Recall the rotated $u-v$ plane. The $\alpha-$level curves are given by $\left(u, g_\alpha(u)\right)$ for $|u| \leq \sqrt{2\alpha(1-\alpha)}$, where $g_\alpha(u)$ is defined as in \ref{eq: 6} for various values of $\alpha$. Define the function
\[
G(a, u) : = \frac{2u}{\pi}\tan^{-1}\left( \frac{\sqrt{2}au}{\sqrt{1-a^2-2u^2}}\right) + \frac{\sqrt{2}}{\pi}\tan^{-1}\left( \frac{\sqrt{1-a^2-2u^2}}{a} \right).
\]Note that this is the function $g_\alpha(u)$ evaluated at $\alpha = \frac{1-a}{2}$ for $0<a<1$. Therefore, in the rotated system, the area below this curve, which is within the rotated square (refer to Figure \ref{fig:rot_lvl}), is given by
\[
I(a) : = \int_{-A(a)}^{A(a)} G(a, u)du \, - \, \frac{1-a^2}{2}.
\] where $\displaystyle A(a) = \sqrt{\frac{1-a^2}{2}} $. Observe that this area is part of the excluded region. Similarly, the area in the top right corner (figure \ref{fig:pre_R}), above the curve $L(x, y) = \frac{1+a}{2}$, is also given by $I(a)$ due to the following reflection property \cite{lim_shapes}:
\[
L(x, y) = 1-L(1-x, 1-y).
\]
In addition, the excluded areas corresponding to the parameter $b$, the remaining two corners, are also given by $I(b)$ (together, $2I(b)$) as can be argued by reflecting back the curves $L(1-x, y) = \frac{1\pm b}{2}$ along the line $x = 1/2$. Therefore, the measure of rectangle $R$, that is, the area bounded by the curves, is 
\[
1- 2\left( I(a) + I(b)\right), 
\] where $a$ and $b$ are related by the equation \ref{eq:7}. Since we want to maximize this expression, we must minimize $I(a) + I(b)$. Equation \ref{eq:7} can also be described as
\begin{align}\label{eq:8}
    \sqrt{1-a^2} \, + \, \sqrt{1-b^2} = 1.
\end{align}
   To minimize $I(a) + I(b)$, routine calculus gives 
\begin{align*}
    \frac{\mathrm{d}I(a)}{\mathrm{d}a} & = \frac{\mathrm{d}}{\mathrm{d}a} \left(\int_{-A(a)}^{A(a)} G(a, u)du\right) \, - \, \frac{\mathrm{d}}{\mathrm{d}a}\left(\frac{1-a^2}{2}\right)  \\
    & = G(a, A(a))\frac{\mathrm{d}A(a)}{\mathrm{d}a} - G(a, -A(a))\left(-\frac{\mathrm{d}A(a)}{\mathrm{d}a}\right) + \int_{-A(a)}^{A(a)}\frac{\partial}{\partial a}G(a, u) du \, + a
\end{align*}
Now, $G(a, A(a)) = G(a, -A(a)) = A(a)$ and $\displaystyle A'(a) = \frac{-a}{2A(a)}$. Thus, 
\begin{align*}
    \frac{\mathrm{d}I(a)}{\mathrm{d}a} & = 2A(a)A'(a) +  \int_{-A(a)}^{A(a)}\frac{\partial}{\partial a}G(a, u) du \, + a 
      =  \int_{-A(a)}^{A(a)}\frac{\partial}{\partial a}G(a, u) du .
\end{align*} 
To compute $\displaystyle \frac{\partial}{\partial a}G(a, u)$, let
\[
G_1(a, u) =  \frac{2u}{\pi}\tan^{-1}\left( \frac{\sqrt{2}au}{\sqrt{1-a^2-2u^2}}\right) \text{ and } G_2(a, u) =\frac{\sqrt{2}}{\pi}\tan^{-1}\left( \frac{\sqrt{1-a^2-2u^2}}{a} \right).
\]
Routine calculations give 
\begin{align*}
   & \frac{\partial G_1(a, u)}{\partial a}  = \frac{2\sqrt{2}u^2}{\pi (1-a^2) \sqrt{1-a^2-2u^2}}, \quad  \frac{\partial G_2(a, u)}{\partial a}  = \frac{-\sqrt{2}}{\pi \sqrt{1-a^2-2u^2}} 
\end{align*}
Thus, we have
\begin{align*}
     \frac{\mathrm{d}I(a)}{\mathrm{d}a} & =  \int_{-A(a)}^{A(a)}\frac{\partial}{\partial a}G(a, u) du \\
      & = \int_{-A(a)}^{A(a)}\left( \frac{2\sqrt{2}u^2}{\pi (1-a^2) \sqrt{1-a^2-2u^2}} - \frac{\sqrt{2}}{\pi \sqrt{1-a^2-2u^2}}  \right)du \\
      & = \frac{-\sqrt{2}}{\pi (1-a^2)}\int_{-A(a)}^{A(a)} \sqrt{1-a^2-2u^2}du.
\end{align*}
Since it is an even function, 
\[
 \frac{\mathrm{d}I(a)}{\mathrm{d}a} =  \frac{-2\sqrt{2}}{\pi (1-a^2)}\int_{0}^{A(a)} \sqrt{1-a^2-2u^2}du = \frac{-1}{2}.
\]
Therefore, 
\[
\frac{\mathrm{d}}{\mathrm{d}a}\left( I(a) + I(b) \right) = \frac{-1}{2}\left( 1+ \frac{\mathrm{d}b}{\mathrm{d}a}\right).
\]
Equation \ref{eq:8} yields $\displaystyle a=b = \frac{\sqrt{3}}{2}$ as the point of global minimum. Thus,
\[
\beta=
\frac{3}{2} - 4\int_{\frac{-1}{2\sqrt{2}}}^{\frac{1}{2\sqrt{2}}}\left[\frac{2u}{\pi}\tan^{-1}\left( \frac{\sqrt{6}u}{\sqrt{1-8u^2}} \right) + \frac{\sqrt{2}}{\pi}\tan^{-1}\left( \frac{\sqrt{1-8u^2}}{\sqrt{3}} \right) \right]du.
\]
Numerical estimates for the above integral yield 
\[
\beta \approx 0.732.
\]

To prove the first part of the theorem for general rectangular permutations of shape $m\times n$ with $m=\theta n$, we use \ref{Gen_tab} and another result similar to theorem \ref{limit_shape} (see \cite{lim_shapes}) about the shape of a random $\pi\in\mathrm{ES}_{\lfloor\theta n\rfloor, n}$ over the domain $[0, 1] \times [0, \theta]$. 
Again, as in the square case, one can compute the level curves $L_\theta:[0,1] \times[0, \theta] \rightarrow[0,1]$, that describe the limiting surface of an $m\times n$ rectangular tableau with side ratio $\theta$. The calculations are more cumbersome, so we omit the details. 
\end{proof}
\section{Efficient Computation of $c(\pi)$: Proof of theorem \ref{nlogn_algo}}
For any $1 \leq i<j\leq n$, observe that \[ c[i, j] = \sum_{k:i<k<j} \mathbb{1}_{\{\pi(i) < \pi(k) < \pi(j)\text{ or } \pi(i) > \pi(k) > \pi(j)\}} + 2.\] and $\displaystyle c(\pi) = \max_{1\leq i < j \leq n} c[i, j]$. 
\begin{definition}[left-to-right minimum and right-to-left maximum]
    Given a permutation $\pi$, any point $\pi_i$ is said to be left-to-right minimum if for every $j < i$, $\pi_j > \pi_i$.  Similarly, a point $\pi_i$ is called right-to-left maximum if for every $j > i$, $\pi_j < \pi_i$. 
\end{definition}
Vacuously, $\pi_1$ is a left-to-right minimum, and $\pi_n$ is a right-to-left maximum.\\

For a permutation $\pi$, let $L_\pi := \{   a_1, \, a_2, \dots, a_{\ell} \}$ be the set of all left-to-right minima, and $R_\pi : = \{ b_1, \, b_2, \dots, b_r  \}$, be the set of all right-to-left maxima. Here, $a_i, b_j$ denote the positions of the left-to-right minima and the positions of the right-to-left maxima, respectively, both in increasing order. 
\begin{lemma}\label{alg_lemma}
    Any maximum forward caged sequence (FCS) has its endpoints in $L_\pi$ and $R_\pi$, respectively. 
\end{lemma}
\begin{proof}
    Let $i, \, j \in [n]$ be such that $\pi_i < \pi_j$ and $c[i, j]$ is the maximum length of any FCS in $\pi$ (not necessarily $c(\pi)$). If the point $\pi_i$ is not a left-to-right minimum, there exists some $k < i$ with $\pi_k < \pi_i$. Then $\pi_k < \pi_i < \pi_j$ and $c[k, j] \geq c[i, j] + 1 $ (both FCS), contradicting the maximality of $c[i, j]$. A similar argument also holds for the right endpoint, proving the lemma. 
\end{proof}
By definition, $\pi(a_1)>\pi(a_2) > \cdots > \pi(a_l)$ and $\pi(b_1) > \pi(b_2) > \cdots > \pi(b_r)$.  Fix $1\le k\le n$. Call $a_i$ a \textit{valid $k-$point} of $L_\pi$ if $a_i \leq k$ and $\pi(a_i) \leq \pi(k)$. Since every $j > i$ with $a_j \leq k$, $\pi(a_j) \leq \pi(k)$ is also a valid $k-$point, it follows that the collection of all valid $k-$points forms a contiguous block of $L_\pi$. Similarly, call a point $b_j \in R_\pi$ a valid $k-$point if $b_j \geq k$ and $\pi(b_j) \geq \pi(k)$. Again, it follows that all valid $k-$points form a contiguous block of $R_\pi$ as well. We will also refer to the index $i$ as a \textit{valid $k-$index}, if the corresponding point ($a_i$ or $b_i$ as the case may be) is a valid $k-$point. We shall call a pair $(a_i,b_j)$ a \textit{valid $k-$pair} if both $a_i$ and $b_j$ are valid $k-$points of $L_{\pi}$ and $R_{\pi}$ respectively.  We say that $(a_i, b_j)$ \textit{cages a point $(k, \pi(k))$} if $a_i\leq k \leq b_j$ and $\pi(a_i) \leq \pi(k) \leq \pi(b_j)$. In other words, $a_i$ (respectively, $b_j$) is a valid $k-$point in $L_\pi$ (respectively, in $R_\pi$). We are now in a position to state the algorithm:

\textbf{INPUT:} $\pi \in S_n$. \\
\textbf{OUTPUT:} $i, \, j \in [n]$ such that $c(\pi) = c[i, j]$, and also $c(\pi)$. \\
\textbf{ALGORITHM:} 
\begin{enumerate}
    \item Obtain $L_\pi = \{a_1, a_2, \dots, a_{\ell}\}$ and $R_\pi = \{ b_1, b_2, \dots, b_r\}$. \label{initial}
    \item For every permutation point $(k, \, \pi(k))$,
    \label{2}
    \begin{enumerate}
        \item Find the maximum $a_{i_0} \in L_\pi$ such that $a_{i_0} \leq k$. Define $\ell_{2, k} : = i_0$. \label{2a}
        \item In the interval $[1, \, \ell_{2,k}]$, find the minimum $i_1$ with $\pi(a_{i_1}) \leq \pi(k)$. Let $\ell_{1,k} : = i_1 $. \label{2b}
        \item Find the smallest $b_{j_0} \in R_\pi$ such that $b_{j_0} \geq k $. Let $r_{1,k} :=j_0 $. \label{2c}
        \item In the interval $[r_{1,k}, \, r] $, find the maximum $j_1$ with $\pi(b_{j_1}) \geq \pi(k) $. Define $r_{2,k} := j_1 $.\label{2d}
    \end{enumerate}
    Define $\mathcal{R}_k : = [\ell_{1,k}, \, \ell_{2,k}] \times [r_{1,k}, \, r_{2,k}] \subset [\ell]\times [r] $.
    \item Find $(i, j) \in [\ell] \times [r]$ such that $c[a_i, b_j] = \# \left\{ k \in [n] : (i, j) \in \mathcal{R}_k \right\}$ is maximized. Return such pairs $(i, j) $ along with $c[a_i, b_j]$. \label{final}
    \item Repeat steps \ref{initial} to \ref{final} for $\pi^R$. Compare the two outputs and report the maximum along with the corresponding pair/pairs $(a_i, b_j)$. 
\end{enumerate}
\begin{proof}[Proof of correctness and time complexity]
    After initialization, to verify if a point is in $L_\pi$ (respectively, in $R_\pi$), we simply compare it with the last added entry in the respective arrays; consequently, both $L_\pi$ and $R_\pi$ can be obtained in just two passes of $\pi$ (one from left to right and second from right to left), ergo, in $O(n)$ time. \\
    
    By lemma \ref{alg_lemma}, any maximum FCS is of the form  $[a_i,b_j]$ with $a_i \in L_\pi$ and $b_j \in R_\pi$, so it suffices to restrict our search to pairs from $L_\pi \times R_\pi$. \\
    
    Fix $1\le k\le n$. suppose step \ref{2a} returns $i_0$. If $a_{i_0}$ is not a valid $k-$point, then we must necessarily have $\pi(a_{i_0}) > \pi(k)$, which implies that either $k$ is a left-to-right minimum or that there exists some $i>i_0$ with $a_i < k$ and $\pi(a_i) < \pi(k)$, both contradicting the output of step \ref{2a}. Now, since $i_0$ is a valid $k-$index, step \ref{2a} ensures that it must be the rightmost valid $k-$point in $L_\pi$. The next step (step \ref{2b}) gives the other end of the corresponding contiguous block. By similar arguments (vis-\`a-vis  steps \ref{2c} and \ref{2d}), and the fact that the $a_i$'s, $\pi(a_i)$'s, $b_j$'s and $\pi(b_j)$'s are in sorted order, binary search can be applied in steps \ref{2a} - \ref{2d} to get the box $\mathcal{R}_k$ in $O(\log n)$. Since the algorithm performs this for each $k\in [n]$, the overall time complexity of step \ref{2} is $O(n\log n)$. \\
    
    Once we have all the rectangles $\mathcal{R}_k$, we want to find those pairs $(i, j) \in [\ell] \times [r]$ for which the corresponding $a_i$ and $b_j$ cage the most number of points $k \in [n]$, i.e.,  we seek those pairs $(i, j)$ that lie in the most number of rectangles from the family $\{ \mathcal{R}_k: 1\le k\le n\}$. This is the standard maximum overlapping rectangle problem that can be solved efficiently in $O(n\log n)$ time using the sweep-line algorithm augmented with a max-segment tree (see \cite{sweeping-1} \& \cite{sweeping}) as follows. \\ 
    
     Since all the rectangles $\mathcal{R}_k$  are axis-parallel, keep track of their horizontal and vertical boundary intervals, \textit{viz.} $[\ell_{1,k}, \, \ell_{2,k}]$ and $ [r_{1,k}, \, r_{2,k}]$, respectively, separately. 
     Collect and sort all horizontal interval endpoints in \(O(n\log n)\) time and use them to construct a segment tree \cite{sweeping} endowed with lazy propagation. Each leaf represents an elementary interval between consecutive endpoints, while each internal node - which corresponds to a union of the elementary intervals that correspond to the leaves below that node - stores the maximum overlap in its interval, along with the location of a leaf where such a maximum is attained. Lazy propagation enables efficient range updates without explicitly updating every leaf. \\

Next, sort the vertical coordinates of the rectangle boundaries and sweep a vertical line through them in increasing order. Each rectangle generates two events: an \emph{insertion} when the sweep reaches its lower boundary and a \emph{deletion} when it reaches its upper boundary. At each event, perform a range addition of \(+1\) or \(-1\), respectively, over the rectangle's horizontal interval. The segment tree maintains the maximum overlap and a corresponding location, which can be retrieved from its root. Since each rectangle generates two events and each range update takes \(O(\log n)\) time, the overall time complexity is \(O(n\log n)\). Unlike the algorithm in \cite{sweeping}, we do not report intersecting pairs; we only maintain the maximum overlap. \\
     
    This algorithm evaluates the maximum length of any FCS in $\pi$. For BCS in $\pi$, the same algorithm is run over the reverse permutation $\pi^R$ (which can be described in linear time),  yielding $c(\pi)$, in run time of $O(n\log n)$. 
\end{proof}
\subsection{Randomized linear time approximation}
Before we state the algorithm, we describe the main idea. In the previous algorithm, only steps \ref{2} and \ref{final} have time complexity $O(n \log n)$. Indeed, if we recall the analysis, in step \ref{2}, determining $\mathcal{R}_k$ runs in $O(\log n)$ for each $k \in [n]$. Similarly, the bound on the time complexity for step \ref{final} is determined by the number of rectangles $\mathcal{R}_k$. So, instead of performing these steps for all $k \in [n]$, we choose a random choose $S \subset [n]$, and restrict these steps to $S$, thereby making the improvement. We elaborate on this below.

\textbf{INPUT:} Any permutation $\pi \in S_n$. \\
\textbf{OUTPUT:} $\widehat{c(\pi)}$ such that $\left| c(\pi) - \widehat{c(\pi)}  \right| \lesssim O\left( \sqrt{n} \log n\right)$ with high probability.  

\textbf{ALGORITHM:} 
\begin{enumerate}
    \item Obtain $L_\pi = \{a_1, a_2, \dots, a_l\}$ and $R_\pi = \{ b_1, b_2, \dots, b_r\}$. \label{initial_r}
    \item Set $\displaystyle p = \frac{1}{\log n}$. Choose a $p-$random subset $S \subset [n]$, i.e., pick each $x\in [n]$ into $S$ independently, with probability $p$.   \label{random_step}
    \item For every $k \in S$, do the following: \label{2_r}
    \begin{enumerate}
        \item Find the maximum $a_{i_0} \in L_\pi$ such that $a_{i_0} \leq k$. Define $\ell_{2, k} : = i_0$. \label{2a_r}
        \item In the interval $[1, \, \ell_{2,k}]$, find the minimum $i_1$ with $\pi(a_{i_1}) \leq \pi(k)$. Let $\ell_{1,k} : = i_1 $.
        \item Find the smallest $b_{j_0} \in R_\pi$ such that $b_{j_0} \geq k $. Let $r_{1,k} :=j_0 $. \label{2c_r}
        \item In the interval $[r_{1,k}, r] $, find the maximum $j_1$ with $\pi(b_{j_1}) \geq \pi(k) $. Define $r_{2,k} := j_1 $.\label{2d_r}
    \end{enumerate}
    Define $\mathcal{R}_k : = [\ell_{1,k}, \, \ell_{2,k}] \times [r_{1,k}, \, r_{2,k}] \subset [l]\times [r] $.
    \item Find $(i, j) \in [\ell] \times [r]$ such that $c_S[a_i, b_j] := \# \left\{ k \in S : (i, j) \in \mathcal{R}_k \right\}$ is maximized. 
    \item Define $\widehat{c[a_i, b_j]} : = \log n \cdot c_S[a_i, b_j]$. Return such pairs $(i, j) $ along with $\widehat{c[a_i, b_j]}$. \label{final_r}
    \item Repeat steps \ref{initial_r} to \ref{final_r}, except step \ref{random_step} (take $S' = \{ n + 1 - i : i \in S\}$ instead), for $\pi^R$. Compare the two outputs and report the maximum along with the corresponding pair/pairs $(a_i, b_j)$. 
\end{enumerate}
\begin{proof}[Analysis of algorithm]
    Since $S$ is a $1/\log n-$random subset of $[n]$, the expected size of $S$ is $\mathbb{E}|S| = n/\log n$, and furthermore, by standard Chernoff bounds, $|S|\le 2n/\log n$ \textit{whp}. Therefore, by the previous result, the expected run time of the algorithm is $O(n)$.\\
    We will now analyze the algorithm for $\pi$. Define, for all $(i, j) \in [l] \times [r]$,
    \[
    X_{i,j} : = \left\{ k \in S : (i, j) \in \mathcal{R}_k \right\}.\] 
    
    
    Set $\displaystyle\epsilon = \frac{7\log n}{\sqrt{n}}$. Note that each $X_{i, j} $ is distributed as $\displaystyle \textrm{Bin}\left(c[a_i, b_j],1/\log n  \right)$ and hence $\mathbb{E}\left[ X_{i,j} \right] = c[a_i, b_j]\log n$. Therefore, by the Chernoff-type bounds for the Binomial distribution (see \cite{JLR}, Theorem 2.1)
    \begin{align*}
        \mathbb{P}\left( X_{i, j} \geq (1 + \epsilon)pc(\pi) \right) & = \mathbb{P}\left(  X_{i,j} \geq pc[a_i, b_j] + (1 + \epsilon)pc(\pi)  - pc[a_i, b_j] \right) \\
        & \leq \exp{\left(-\frac{\left\{(1+\epsilon)pc(\pi) - pc[a_i, b_j]\right\}^2}{2\left( pc[a_i, b_j] + \frac{(1+\epsilon)pc(\pi) - pc[a_i, b_j]}{3}\right)}\right)} \\
        & =  \exp{\left(-\frac{3p}{2} \cdot \frac{\left\{(1+\epsilon)c(\pi) - c[a_i, b_j]\right\}^2}{\left( (1+\epsilon)c(\pi)  + 2c[a_i, b_j]\right)}\right)} \\
        & \leq  \exp{\left(-\frac{3p}{2} \cdot \frac{\left\{\epsilon c(\pi) + c(\pi) - c[a_i, b_j]\right\}^2}{ (3+\epsilon)c(\pi) }\right)} \\
       & \leq \exp{\left(-\frac{3p}{2} \cdot \frac{\epsilon^2 c(\pi)}{ (3+\epsilon) }\right)} \\
       & \leq \exp{\left(-\frac{3p\epsilon^2 n}{40} \right)} < \frac{1}{n^3}
    \end{align*} for sufficiently large $n$. 
    In the last step, we have used the lower bound on $c(\pi)$ by Theorem \ref{caged}. Since the above calculation was for one pair $(a_i, b_j)$, by the union bound
    \[
    \mathbb{P}\left( \exists \text{ a pair $(i, j) \in [l] \times [r]$ such that } X_{i,j} \geq (1+\epsilon)pc(\pi) \right) < \frac{1}{n}.
    \] Therefore, with probability $1 -\frac{1}{n}$, $\widehat{c(\pi)} \leq (1+\epsilon)c(\pi)$. To prove the other way around, we just need to find one pair $(i, j)$ such that $X_{i, j} \geq (1-\epsilon)c(\pi) $. Without loss of generality, let $(i, j) \in [l] \times [r]$ be a pair with $c(\pi) = c[a_i, b_j]$. Thus, for this specific pair,  
    \begin{align*}
        \mathbb{P}\left( X_{i, j} \leq (1 - \epsilon)pc(\pi) \right) & = \mathbb{P}\left( X_{i, j} \leq (1 - \epsilon)pc[a_i, b_j] \right) \\
        & =  \mathbb{P}\left( X_{i, j} \leq pc[a_i, b_j] - \epsilon pc[a_i, b_j] \right) \\
        & \leq \exp{\left( -\frac{\epsilon^2p^2(c[a_i, b_j])^2}{2pc[a_i,b_j]} \right)} \\
        & = \exp{\left( -\frac{\epsilon^2pc[a_i, b_j]}{2} \right)} \\
        & \leq \exp{\left( -\frac{\epsilon^2n}{10\log n} \right)} < \frac{1}{n^4}.
    \end{align*}
    It follows that with probability $1-\frac{1}{n^4}$, $ \widehat{c(\pi)} \geq  (1-\epsilon) c(\pi)$. Finally, we have, with probability $\displaystyle 1 - \frac{2}{n}$, 
    \[
      \left| \widehat{c(\pi)}  - c(\pi) \right| \leq \epsilon c(\pi) \sim O\left( \sqrt{n}\log n \right),
    \] and that completes the proof.
\end{proof}
\section{Concluding remarks and some open questions}
\begin{itemize}
\item The set $\mathcal{E}_{ext}$ of \emph{all} the extremal permutations that minimize $c(\pi)$ is not yet completely known. It is of course clear that $\mathcal{E}_{ext}$ is closed under the operations $\pi\to\pi^R,\pi\to\pi^{-1}$ and $\pi\to \overline{\pi}$ but whether there are permutations distinctly different from $\pi_{ext}$ is not clear. Even within this family, some of them are more optimal than others. For instance, recall the permutation $\pi_{ext}$. For $n = 5k + 4$, its middle interval has length $k$ and all other intervals have length $k + 1$. Since this middle interval entirely contributes to $c[\bar{\mathbf{1}}, \bar{\mathbf{5}}]$ as well as $c[\bar{\mathbf{2}}, \bar{\mathbf{4}}]$, the number of maximal caged sequences is $O(k^2)$. If $n \not\equiv -1 \mod 5$, the length of the middle interval can be reduced, which brings the number of optimal caged sequences down to $O(k)$. However, whether such examples exist for $n = 5k + 4$ is not yet clear. 
\item A simple corollary of the ideas in the proof of theorem \ref{caged} is the following: Given $\lfloor n/5\rfloor+2\le\ell\le n$ there exists $\pi\in S_n$ such that $c(\pi)=\ell$. To see why, partition the interval $[1, n]$ in $5$ parts $n_1$ to $n_5$ (some of them can be empty). Recall the ``interval" definition of $\pi_{ext} $. Let $n_i \uparrow$ (respectively, $n_i \downarrow$) indicates interval $n_i$ in increasing (respectively, decreasing) order. Define the permutation $$\pi = n_4 \uparrow n_1\downarrow n_3 \uparrow n_5 \downarrow n_2 \uparrow $$ Observe that its structure is similar to $\pi_{ext} $ except the lengths of the intervals are variable. So, we have $\max \text{FCS}(\pi) = \max\left\{ |n_3| + \mathbb{1}_{n_1 \neq \emptyset} + \mathbb{1}_{n_5 \neq \emptyset} , \, |n_4| + \mathbb{1}_{n_5 \neq \emptyset}, \, |n_2| + \mathbb{1}_{n_1 \neq \emptyset} \right\}$ while $\max \text{BCS}(\pi) = \max\left\{ |n_3| + \mathbb{1}_{n_4 \neq \emptyset} + \mathbb{1}_{n_2 \neq \emptyset}, \, |n_1| + \mathbb{1}_{n_4 \neq \emptyset}, \, |n_5| + \mathbb{1}_{n_2 \neq \emptyset}  \right\}$. Now
\[
c(\pi) = \max \left\{ \max \text{FCS}(\pi) , \, \max \text{BCS}(\pi) \right\}, 
\] with $\sum_{i = 1}^5 |n_i| = n $. Clearly, this system is feasible given $c(\pi) = \ell$ for any $\lfloor n/5\rfloor + 2 \leq \ell \leq n$. 
\item A closer inspection of the proofs of theorems \ref{caged} and \ref{asymmetric} shows that it suffices to consider the points $1,\pi^{-1}(1),\pi^{-1}(n), n$ and at least one pair among these will attain the lower bound in the theorem. It is therefore a natural question to ask the following: For a given $\pi$, if $[i_{\pi},j_{\pi}]$ is an interval satisfying $c[i_{\pi},j_{\pi}]=c(\pi)$, then are $i_{\pi}$ or $j_{\pi}$ in a close neighborhood of any one of those four points? It turns out that such is not the case, by the following simple argument. Let $n=4k^2+4$. Pick $\sigma$ uniformly randomly from $\mathrm{ES}_{2k}$ on the set $[4k^2+2] \setminus \{2k^2 +1, 2k^2 + 2 \}$ and define $\pi\in S_n$ by setting $\pi(1)=2k^2+1,\pi(n)=2k^2+2, \pi(2k^2+1)=1, \pi(2k^2+2)=n$ and $\pi(i)=\sigma(i-1)+1$ for $2\leq i< 2k^2 + 1$, and $\pi(i) = \sigma(i-3) + 1$ for the remaining $i$. It is easy to see that $c(\pi)$ will be more or less determined by the maximal interval for $\sigma$. Moreover, it follows from the proof of theorem \ref{sqaure} that the ends of any maximal interval for $\sigma$ occur far away from any of the four distinguished points for $\pi$.
\item We have already seen that restricting the number of runs in $\pi$ does not necessarily alter the statement of theorem \ref{caged} in any significant manner. One can similarly impose algebraic restrictions on $\pi$. If we insist that $\pi$ comes from a much smaller subgroup of $S_n$ does this change the answer much? Again, note that in fact, $\pi_{ext}$ is in fact in the hyperoctahedral group $B_n$ which is a subgroup of $S_n$ (for even $n$) of size $2^{n/2}(n/2)!$ which is a much smaller group than $S_n$, so merely restricting the size of the subgroup does not immediately give us a bigger answer even if the restricted group is exponentially smaller than $S_n$. Another restriction one may impose is to have $\pi$ admit a large cycle in its cycle decomposition. While $\pi_{ext}$ itself is a permutation of order four, a slight tinkering with $\pi_{ext}$ allows us to get a permutation $\pi\in S_n$ with $c(\pi)\sim n/5$ and with $\pi$ admitting a $(4n/5)$-sized cycle in its cycle decomposition. But the problem of determining $c(\pi)$ when we restrict ourselves to the subfamily of $n$-cycles is still open; it is not clear whether this algebraic imposition changes the answer to something significantly larger than $n/5$.
\item Given a pair of permutations $\pi_1,\pi_2\in S_n$, it is a simple consequence of the Erd\H{o}s-Szekeres theorem that there is a set $S$ of size at least $n^{1/2}$ which appear in both, $\pi_1$ and $\pi_2$, in the same relative order, and hence, by theorem \ref{caged}, $\pi_1,\pi_2$ have a commonly caged subset of size at least $\Omega(\sqrt{n})$. In fact, one can get a common caged sequence to both $\pi,\sigma$ of order $\Omega(n)$ as follows. Let $T$ be the total number of triples $(a,x,b)$ such that $x$ is caged between $a,b$ in both $\sigma,\pi$. Define $C(a,b)$ to be the number of $x\in (a,b)$ that are caged by $a,b$ in both $\sigma,\pi$. We have $T=\sum_{a,b} C(a,b)$. This implies that for a random pair $a<b$, $\mathbb{E}C(a,b)=\frac{T}{\binom{n}{2}}$. But note that any subset of $[n]$ of size  $17$ contains a simultaneously caged triple by two iterative applications of the Erd\H{o}s-Szekeres theorem. If $H$ is the $3$-uniform hypergraph on the vertices $[n]$ with $\{a,x,b\}$ being an $\mathcal{H}$-edge if $\{a,x,b\}$ forms a simultaneously caged triple in $\sigma, \pi$. Then $H$ is a $(n,17, 3)$-Tur\'{a}n system, so by De Caen's Tur\'{a}n-system inequality (see, for instance, \cite{deCaen}) we have
    \begin{equation}
        T(n,s,r)\ge \frac{n-s+1}{n-r+1}\frac{\binom{n}{r}}{\binom{r-1}{s-1}}
    \end{equation}
 which, by setting $s=17$ and $r=3$ gives
    \begin{equation}
        \mathbb{E}C(a,b)\ge \frac{T}{\binom{n}{2}}\ge \frac{n-16}{360}.
    \end{equation}
    In fact, the same argument shows that for any $k\ge 2$ and any $\sigma_1,\ldots,\sigma_k\in S_n$, there is a common caged sequence to all the $\sigma_i$ of size at least $\frac{n}{3\cdot2^{2^{k+1}}}$. It now remains an interesting question to determine the correct order of a commonly caged set for any two distinct permutations $\pi,\sigma$.

The size of a largest possible, commonly caged subset of $\pi_1$ and $\pi_2$ - a question of natural interest - is likely to be much larger though that is not clear at the moment.
\item For which shape $\Lambda$ is $\frac{\mathbb{E}_{\pi\sim\Lambda} c(\pi)}{n}$ the least? Here we are picking $\pi$ uniformly from the permutations of shape $\Lambda$.   
\item Experimental data suggests that for certain shapes $\Lambda$, if $\pi$ is picked randomly from the set of permutations of that shape,  $c(\pi)\sim n$ with high probability, while for other shapes, $c(\pi)\le (1-\alpha_{\Lambda}n)$. At the moment, we hazard a conjecture that if the shape $\Lambda$ is \emph{convex} then the latter phenomenon occurs, i.e., $c(\pi)$ is significantly short of $n$. A test case would be to see how $c(\pi)$ behaves for the triangular shapes, i.e., $\Lambda=(k,k-1,\ldots,1)$ for some integer $k$. 
\end{itemize}
\bibliographystyle{amsplain}

\end{document}